\documentclass[11pt]{article}

\usepackage{amsmath, amssymb, amsthm}
\usepackage{dsfont}
\usepackage{graphicx}
\usepackage{tikz}
\usepackage{float}
\usepackage{enumerate}
\usepackage{hyperref}

\usetikzlibrary{calc}
\usetikzlibrary{shapes, positioning}
\usetikzlibrary{backgrounds}

\usetikzlibrary{positioning,calc,fit,backgrounds}

\allowdisplaybreaks
\expandafter\let\expandafter\oldproof\csname\string\proof\endcsname
\let\oldendproof\endproof
\renewenvironment{proof}[1][\proofname]{%
	\oldproof[\bf #1]%
}{\oldendproof}

\allowdisplaybreaks

\theoremstyle{plain}
\newtheorem{lemma}{Lemma}[section]
\newtheorem{theorem}[lemma]{Theorem}
\newtheorem{claim}[lemma]{Claim}
\newtheorem{proposition}[lemma]{Proposition}

\newtheorem{corollary}[lemma]{Corollary}

\newtheorem{definition}[lemma]{Definition}

\newtheorem*{claim*}{Claim}
\newtheorem{fact}[lemma]{Fact}

\theoremstyle{definition}
\newtheorem{example}[lemma]{Example}
\newtheorem{construction}[lemma]{Construction}

\newcommand{\poly}{\text{poly}}
\RequirePackage[normalem]{ulem} 
\RequirePackage{color}\definecolor{RED}{rgb}{1,0,0}\definecolor{BLUE}{rgb}{0,0,1} 

\newcommand{\ex}{\text{ex}}

\newcommand{\girth}{\text{girth}}
\newcommand{\Rem}{\mathrm{Rem}}

\title{NP-Hardness of the $H$-Free Edge-Deletion Problem}
\author{Lior Gishboliner\thanks{Department of Mathematics, University of Toronto, Canada.
\emph{Email}: \href{mailto:lior.gishboliner@utoronto.ca}{\tt lior.gishboliner@utoronto.ca}. Research supported by an NSERC Discovery Grant.
}
\and Ethan Honest\thanks{Department of Mathematics, University of Toronto, Canada.
\emph{Email}: \href{mailto:lior.gishboliner@utoronto.ca}{\tt ethan.honest@mail.utoronto.ca}.}
}

\date{}

\begin{document}

\maketitle

\begin{abstract}
    \noindent
    For a graph $H$, the {\em $H$-freeness edge-deletion problem} is the algorithmic problem of finding, for an input graph $G$, the minimum number of edges of $G$ whose deletion turns $G$ into an $H$-free graph. We show that for every graph $H$ containing a cycle, this problem is NP-hard. This proves a conjecture of Gishboliner, Levanzov and Shapira, and completes the characterization of the complexity of the $H$-freeness edge-deletion problem, answering a question of Alon, Shapira \nolinebreak and \nolinebreak Sudakov. 
\end{abstract}

\section{Introduction}\label{sec:intro}

Edge-deletion problems are a class of algorithmic problems where the goal is to find the minimum number of edges whose deletion turns the input graph $G$ into a graph satisfying a given graph property $\mathcal{P}$. The systematic study of such problems was initiated by Yannakakis \cite{Yannakakis}, 
who proved the NP-hardness of this problem for several natural graph properties $\mathcal{P}$, and raised the question of proving hardness for large families of graph properties. 

For a (non-empty) graph-family $\mathcal{H}$ each containing at least one edge, and for a graph $G$, we denote by $\Rem_{\mathcal{H}}(G)$ the distance of $G$ to $\mathcal{H}$-freeness, i.e., the minimum number of edges whose deletion turns $G$ into an $\mathcal{H}$-free graph.\footnote{A graph is {\em $\mathcal{H}$-free} if it contains no (not necessarily induced) copies of any $H \in \mathcal{H}$.}
If $\mathcal{H} = \{H\}$, then we write $\Rem_H(G)$. 
A breakthrough result of Alon, Shapira and Sudakov \cite{AlonShapiraSudakov} addressed Yannakakis's question by proving that for every (possibly infinite) family of graphs $\mathcal{H}$ containing no bipartite graph, it is NP-hard to compute $\Rem_{\mathcal{H}}(G)$, and even to approximate this quantity to within an additive error of $n^{2-\varepsilon}$ for $n$-vertex input graphs $G$, for any fixed $\varepsilon > 0$. In the case that $\mathcal{H}$ is finite, this result was reproved and extended to $k$-uniform hypergraphs in \cite{HypergraphEdgeModification}. 

Alon, Shapira and Sudakov \cite{AlonShapiraSudakov} raised the problem of characterizing the graphs $H$ for which the problem of computing $\Rem_H(\cdot)$ is solvable in polynomial time. As mentioned above, their results show that this problem is NP-hard if $H$ is non-bipartite. However, their technique does not apply to bipartite $H$.\footnote{Furthermore, the aforementioned inapproximability result of \cite{AlonShapiraSudakov} does not hold for bipartite $H$, because for every bipartite $H$ there exists $\varepsilon > 0$ such that every $n$-vertex $H$-free graph has at most 
$n^{2-\varepsilon}$ edges (by the K\H{o}v\'ari-S\'os-Tur\'an theorem \cite{KST}). Hence, one can trivially approximate $\Rem_H(G)$ up to additive error of $n^{2-\varepsilon}$ by outputting $e(G)$.} Earlier, 
Yannakakis \cite{Yannakakis} proved that computing $\Rem_H(\cdot)$ is NP-hard if $H$ is a cycle (see also \cite{AlonStav}), and
Watanabe, Ae and Nakamura \cite{WatanabeAeNakamura} proved hardness in the case that $H$ is 3-connected.
Levanzov, Shapira and the first author \cite{gishboliner2024trimming} considered the case where $H$ is a forest, and showed that computing $\Rem_H(\cdot)$ is polynomial-time solvable if $H$ is a star forest (i.e., a forest where every component is a star) and NP-hard otherwise. They conjectured that star forests are the only graphs $H$ for which the problem is polynomial. In other words, what remains is to show that computing $\Rem_H(\cdot)$ is NP-hard whenever $H$ contains a cycle. In this paper we prove this conjecture. This also answers the aforementioned question of Alon, Shapira and Sudakov \cite{AlonShapiraSudakov}.

\begin{theorem}\label{thm:main}
For every graph $H$ containing a cycle, the problem of computing $\Rem_H(\cdot)$ is NP-hard.
\end{theorem}

\begin{corollary}\label{cor:main}
    For a graph $H$ containing at least one edge, computing $\Rem_H(\cdot)$ is polynomial-time solvable if $H$ is a star forest, and is NP-hard otherwise.
\end{corollary}

It is worth noting that the analogous problem for {\em induced} $H$-freeness was solved previously by Aravind, Sandeep and Sivadasan \cite{Induced} using different methods. Here the task is to make the input graph induced $H$-free using the smallest number of edge deletions, additions, or both; these are known as the induced edge-deletion, edge-addition, and edge-modification problems, respectively. 

We prove Theorem \ref{thm:main} via a reduction from the vertex cover problem. This is natural, since the vertex cover problem asks for the smallest number of vertex-deletions which leave the graph edgeless, while $\Rem_H(\cdot)$ asks for the smallest number of edge-deletions which leave the graph $H$-free. Thus, roughly speaking, the reduction works by taking an input graph $G$ for the vertex cover problem, and producing a graph $A = A(G,H)$ which has special edges corresponding to vertices of $G$, and such that each edge $e = xy \in E(G)$ gives rise to a copy of $H$ in $A$ involving the two special edges corresponding to $x,y$. We note that this idea is not new; e.g., it was used by Alon and Stav \cite{AlonStav} to show that computing $\Rem_H(\cdot)$ is NP-hard if $H$ is a cycle. However, analyzing this construction becomes significantly more difficult when dealing with a general graph $H$. The difficulty lies in controlling the copies of $H$ in the graph $A = A(G,H)$. This is essential in order to show that by taking a vertex-cover $C$ of $G$ and deleting all special edges corresponding to the vertices in $C$, one turns $A$ into an $H$-free graph. As we shall see, this statement is the main challenge in the proof of Theorem \ref{thm:main}.

\paragraph{Paper organization:} Section \ref{sec:prelim} contains some preliminaries. Section \ref{sec:construction} introduces the main construction used in our reduction and proves some of its key properties. We then prove Theorem \ref{thm:main} in several steps: first for 2-connected graphs $H$ in Section \ref{sec:2-connected}, then for connected graphs $H$ with minimum degree $\delta(H) \geq 2$ in Section \ref{sec:connected, delta >=2}, and finally in full generality in Section \ref{sec:main}.

\paragraph{Notation:} We use standard graph-theoretic notation. In particular, $\delta(G)$ denotes the minimum degree of $G$ and $\tau(G)$ the vertex-cover number of $G$. Also, $G[X]$ denotes the subgraph of $G$ induced on a set $X \subseteq V(G)$.

\paragraph{Declaration of Use of AI:} The proofs in this paper were found with no significant use of AI. ChatGPT was used to handle some minor issues and polish the writing.

\section{Preliminaries}\label{sec:prelim}

We need the following known fact, stating that the vertex-cover problem remains hard for input graphs of large girth. This follows e.g.~from \cite{Murphy}.
\begin{lemma}\label{thm:high-girth-VC-is-np-hard}
    For every fixed $g \geq 3$, the vertex-cover problem is NP-hard for input graphs with girth at least $g$.
\end{lemma}
     

\noindent
We also recall the notion of a separator.


\begin{definition}
    Let $G$ be a graph. A set $S \subseteq V(G)$ {\em separates} two subsets 
    $A,B \subseteq V(G)$ if there is no $A,B$-path in $G-S$.\footnote{In particular, $A \setminus S$ and $B \setminus S$ are disjoint.}
    We say that $S$ is a {\em separator} of $G$ if $G - S$ is disconnected.\footnote{Equivalently, $S$ is a separator if it separates some two non-empty sets $A,B \subseteq V(G) \setminus S$.} 
\end{definition}

\enlargethispage{\baselineskip}

\noindent
Observe that a separator of $G$ can be empty.
Next, we recall the definition of $k$-connectivity.



\begin{definition}
    A graph $G$ is called {\em $k$-connected} if $|V(G)| \geq k+1$ and $S \subseteq V(G)$ is not a separator of $G$ whenever $|S| < k$.
\end{definition}
Two paths are called internally vertex-disjoint if they share no internal vertices.
We will frequently use Menger's theorem: A graph $G$ with $|V(G)| \geq 2$ is $k$-connected if and only if for every pair of distinct vertices $u,v \in V(G)$, there are at least $k$ distinct internally vertex-disjoint $u,v$-paths. An equivalent version is as follows: If 
$|V(G)| \geq k+1$, then $G$ is $k$-connected if and only if for every pair of distinct non-adjacent vertices $u,v \in V(G)$, there are at least $k$ internally vertex-disjoint $u,v$-paths.
We will also use the so-called fan lemma: In a $k$-connected graph, for every vertex $v$ and set $S$ of size $k$ with $v \notin S$, there are $k$ paths from $v$ to $S$ which only intersect at $v$.


We also need the following simple lemma on $k$-connectivity. For graphs $G_1,G_2$, the union $G_1 \cup G_2$ is defined as the graph with vertex-set $V(G_1) \cup V(G_2)$ and edge-set $E(G_1) \cup E(G_2)$.
\begin{lemma}\label{claim:generalized-glue-on-edges}
    Let $k \geq 1$, and let $G_1, G_2$ be $k$-connected graphs with $|V(G_1) \cap V(G_2)| \geq k$. Then, $G := G_1 \cup G_2$ is $k$-connected. 
\end{lemma}

\begin{proof}
    First, observe that $|V(G)| \geq |V(G_1)| \geq k + 1$. Now fix any $S \subseteq V(G)$ with $|S| < k$. By the assumption that $G_1$ and $G_2$ are $k$-connected, we have that $G_1 \setminus S$ and $G_2 \setminus S$ are connected.
    Also, $(V(G_1) \cap V(G_2)) \setminus S \neq \emptyset$ as $|V(G_1) \cap V(G_2)| \geq k > |S|$. Let $v \in (V(G_1) \cap V(G_2)) \setminus S$. We claim that in $G \setminus S$, every vertex is connected by a path to $v$. Indeed, since $G_1 \setminus S$ is connected and $v \in G_1 \setminus S$, each vertex $w_1 \in G_1 \setminus S$ is connected to $v$ by a path in $G_1 \setminus S$. Similarly, each vertex $w_2 \in G_2 \setminus S$ is connected to $v$ by a path in $G_2 \setminus S$. 
    So all vertices of $G \setminus S$ are in the same connected component as $v$, as required.
\end{proof}

\section{The main construction}\label{sec:construction}

As mentioned in Section \ref{sec:intro}, our strategy for proving the NP-hardness of computing $\Rem_H( \cdot )$ (for every graph $H$ containing a cycle) is via a reduction from the vertex cover problem. We now introduce the primary construction used in the reduction. (We note that for some graphs $H$, Construction \ref{construction:general-construction} is not suitable and we will instead use a variant thereof; see Construction \ref{construction:second construction}. The reader may ignore this point for the time being.)

\begin{construction}[$A(G,H,s,a,b)$]\label{construction:general-construction}
    Let $G,H$ be graphs, where $V(G) = [n]$, and let $s,a,b \in V(H)$ be distinct vertices with $sa,sb \in E(H)$. Define a graph $A = A(G, H, s, a, b)$ as follows: 
    \begin{itemize}
    \item Start with vertices $\{s',w_1,\dots,w_n\}$ and edges $s'w_i$ for $i \in [n]$.
    \item For each edge $ij \in E(G)$ with 
            $i < j$, 
            add a copy $H_{ij}$ of $H$ in which $s'$ plays the role of $s$, $w_i$ plays the role of $a$, $w_j$ plays the role of $b$, and all other vertices are new (i.e., are disjoint from $\{s',w_1,\dots,w_n\}$ and from all other such copies). Denote by 
            $f_{ij}: H \rightarrow H_{ij}$ the corresponding isomorphism, so that 
            $f_{ij}(s) = s'$, 
            $f_{ij}(a) = w_i$,
            $f_{ij}(b) = w_j$. 
            We will sometimes refer to the copies $H_{ij}$ as the {\em basic copies} of $H$ in $A$, and refer to $V(H_{ij}) \setminus \{s',w_i,w_j\}$ as the {\em internal \nolinebreak vertices} \nolinebreak of \nolinebreak $H_{ij}$. 
            
    \end{itemize}
\end{construction}

See Figure \ref{fig:A construction} for an illustration of Construction \ref{construction:general-construction}.
Before proceeding, we note several basic properties of the construction. 
\begin{fact}\label{fact:outside vertices}
    Let $A = A(G,H,s,a,b)$.
    For every $ij \in E(G)$, the vertices in $V(H_{ij}) \setminus \{s',w_i,w_j\}$ have no neighbors in
    $V(A) \setminus V(H_{ij})$.
\end{fact}
\begin{proof}
    This is immediate by the definition of the construction.
\end{proof}
\begin{fact}\label{lemma:edges-belong-uniquely}
    Let $A = A(G,H,s,a,b)$. 
    Then for every edge
    $uv \in E(A) \setminus \{s'w_i : i \in [n]\}$, there exists a unique edge
    $ij \in E(G)$ such that $uv \in E(H_{ij})$. 
\end{fact}
\begin{proof}
    First, by construction we have
    $E(A) = \{s'w_i : i \in [n]\} \cup \bigcup_{ij \in E(G)}E(H_{ij})$.
    Therefore, there exists $ij \in E(G)$ with 
    $uv \in E(H_{ij})$.
    To prove uniqueness, suppose that 
    $uv \in E(H_{ij}) \cap E(H_{k\ell})$ for two edges $ij,k\ell \in E(G)$. Observe that
    $$V(H_{ij}) \cap V(H_{k\ell}) = \{s'\} \cup (\{w_i, w_j\} \cap \{w_k, w_{\ell}\}).$$
    As $uv \notin \{s'w_i : i \in [n]\}$, it must be that $\{u,v\} \subseteq \{w_i, w_j\} \cap \{w_k, w_{\ell}\}$.
    Thus, $\{w_i, w_j\} = \{w_k, w_{\ell}\}$ and so $ij = k\ell$, as required. 
\end{proof}

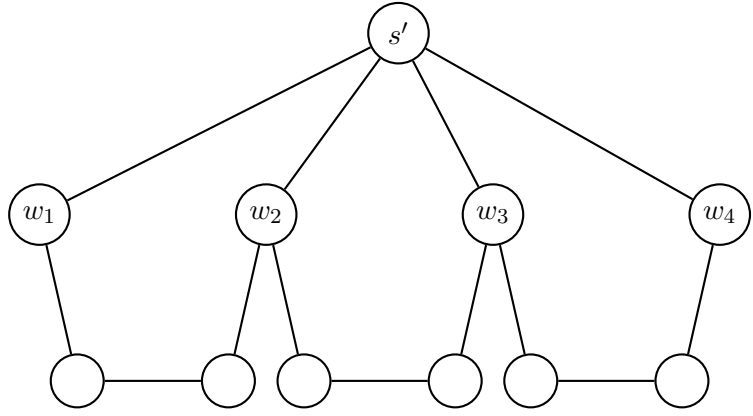
\begin{figure}
    \centering
    \begin{tikzpicture}[
  vertex/.style={circle,draw,thick,minimum size=8mm,inner sep=0pt,fill=white},
  root/.style={vertex,fill=white},
  core/.style={vertex,fill=white},
  private/.style={circle,draw,thick,minimum size=7mm,inner sep=0pt,fill=white,font=\scriptsize},
  edge/.style={draw,thick},
  copybox/.style={draw,black!55,dashed,rounded corners=4pt,inner sep=5pt},
  title/.style={font=\Large\bfseries},
  subtitle/.style={font=\small},
  note/.style={font=\small,align=left},
  insettitle/.style={font=\small\bfseries}
]

\node[root] (sp) at (1.1,3.85) {$s'$};
\node[core] (w1) at (-3.65,1.45) {$w_1$};
\node[core] (w2) at (-0.65,1.45) {$w_2$};
\node[core] (w3) at ( 2.35,1.45) {$w_3$};
\node[core] (w4) at ( 5.35,1.45) {$w_4$};

\draw[edge] (sp)--(w1);
\draw[edge] (sp)--(w2);
\draw[edge] (sp)--(w3);
\draw[edge] (sp)--(w4);

\node[private] (x12) at (-3.15,-0.75) {};
\node[private] (y12) at (-1.15,-0.75) {};
\node[private] (x23) at (-0.15,-0.75) {};
\node[private] (y23) at ( 1.85,-0.75) {};
\node[private] (x34) at ( 2.85,-0.75) {};
\node[private] (y34) at ( 4.85,-0.75) {};

\draw[edge] (w1)--(x12)--(y12)--(w2);
\draw[edge] (w2)--(x23)--(y23)--(w3);
\draw[edge] (w3)--(x34)--(y34)--(w4);
    \end{tikzpicture}
    \caption{The construction $A = A(G,H,s,a,b)$ with $G = \{12,23,34\}$ and $H = C_5$.}
    \label{fig:A construction}
\end{figure}

Next, when $P$ is a subgraph of $G$, we consider the natural identification of $A(P,H,s,a,b)$ with a subgraph of $A(G,H,s,a,b)$.

\begin{fact}\label{fact:subgraph}
    Let $A = A(G,H,s,a,b)$, and let $P$ be a subgraph of $G$. Then $A_P := A(P,H,s,a,b)$
    can be identified with the subgraph of $A$ having vertex-set
    \begin{equation}\label{eq:subgaph of A}
    V(A_P) := \{s'\} \cup \{w_i : i \in V(P)\} \cup \bigcup_{ij \in E(P)}V(H_{ij})
    \end{equation}
    and edge-set
    \begin{equation}\label{eq:subgraph of A, 2}
     E(A_P) := 
    \{s'w_i : i \in V(P)\} \cup \bigcup_{ij \in E(P)}E(H_{ij}).
    \end{equation}
\end{fact}
\noindent
In light of Fact \ref{fact:subgraph}, and with a slight abuse of notation, we will sometimes write 
$A_P = A(P,H,s,a,b)$ for the subgraph of 
$A = A(G,H,s,a,b)$ with vertex-set and edge-set defined in \eqref{eq:subgaph of A} and \eqref{eq:subgraph of A, 2}, respectively. (By Fact \ref{fact:subgraph}, $A_P$ is indeed isomorphic to $A(P,H,s,a,b)$.)

\vspace{0.3cm}


We aim to show that if $s,a,b \in V(H)$ are chosen appropriately and $G$ has large enough girth (larger than $2v(H)$, say), then 
$\Rem_H(A) = \tau(G)$, where $A = A(G,H,s,a,b)$ and $\tau(G)$ is the vertex-cover number of $G$.
We now prove the easier direction of the equality 
$\Rem_H(A) = \tau(G)$. 
\begin{lemma}\label{lemma:easy-inequality}
    Let $G$ be a graph on $[n]$, let $H$ be a graph, and let $s,a,b \in V(H)$ be distinct vertices with $sa,sb \in E(H)$. Let $A = A(G,H,s,a,b)$ as defined in Construction \ref{construction:general-construction}. Then 
    $\Rem_H(A) \geq \tau(G)$.
\end{lemma}


\begin{proof}
Let $S \subseteq E(A)$ such that $A-S$ is $H$-free. Our goal is to show that $|S| \geq \tau(G)$. 
To this end, we first define another edge-set $S' \subseteq E(A)$, as follows. 
For each edge $e \in S$, if $e = s'w_i$ for some $i \in [n] = V(G)$, then add $e$ to $S'$. Otherwise, by Fact \ref{lemma:edges-belong-uniquely}, there is a unique $ij \in E(G)$ such that $e \in E(H_{ij})$. Add either $s'w_i$ or $s'w_j$ to $S'$ arbitrarily. This completes the definition of $S'$. Note that $|S'| \leq |S|$. 
Also, if for some $ij \in E(G)$ it holds that 
$s'w_i, s'w_j \not \in S'$, then 
$E(H_{ij}) \cap S = \emptyset$, which is impossible because $A-S$ is $H$-free. 
Thus, we showed that $C := \{i \in V(G) : s'w_i \in S'\}$ is a vertex-cover of $G$. It follows that $|S| \geq |S'| \geq \tau(G)$, as required. 
\end{proof}

In light of Lemma \ref{lemma:easy-inequality}, we are left with proving the reverse inequality 
$\Rem_H(A) \leq \tau(G)$. The strategy is simple: Take a vertex-cover $C$ of $G$, and delete from $A$ the set of edges $\{ s'w_i : i \in C \}$. The difficulty lies in showing that the resulting graph is $H$-free.\footnote{In most cases, this statement relies on a proper choice of the vertices $s,a,b$.} 
As a demonstration of our approach, we will first prove this statement in the cases that $H$ is 3-connected or a cycle (see Theorem \ref{theorem:main-theorem-H-equals-L} below). These cases are substantially easier than the general case. They also form the induction basis in an inductive proof for the case that $H$ is 2-connected.

In subsequent sections (i.e., Sections \ref{sec:2-connected}-\ref{sec:main}), we will first handle the case of 
2-connected $H$ (see Section \ref{sec:2-connected}) and then connected $H$ with $\delta(H) \geq 2$ (see Section \ref{sec:connected, delta >=2}). In the case that $H$ is connected and $\delta(H) = 1$, one can show that the problem of computing $\Rem_H(G)$ is at least as hard as the problem of computing 
$\Rem_{H_*}(G)$, where $H_*$ is the {\em 2-core} of $H$ (and so $\delta(H_*) \geq 2$). Finally, we will extend the hardness result to disconnected graphs $H$ (containing a cycle). These last two reductions appear in Section \ref{sec:main} and do not require Construction  \ref{construction:general-construction}.

In light of Lemma \ref{thm:high-girth-VC-is-np-hard}, we may assume that $\girth(G) > 2v(H)$, say. 
This is useful due to the following lemma.

\begin{lemma}\label{thm:g-locally-a-forest}
    Let $A = A(G, H, s, a, b)$ and let $H'$ be a copy of $H$ in $A$. 
    Then there exists a subgraph $F \subseteq G$ with $v(F) \leq 2v(H)$ such that $H'$ is a subgraph of $A_{F} = A(F,H,s,a,b)$.
\end{lemma}

\begin{proof}
    We define $F$ as follows: For each $i \in V(G)$ such that $w_i \in V(H')$, put the vertex $i$ into $V(F)$. 
    For each $ij \in E(G)$ such that 
    $w_iw_j \in E(H')$, put the edge $ij$ into $E(F)$ (the vertices $i,j$ are already in $V(F)$ by the previous sentence).
    Finally, for each $ij \in E(G)$ such that 
    $V(H') \cap \nolinebreak (V(H_{ij}) \setminus \nolinebreak \{s',w_i,w_j\}) \neq \nolinebreak \emptyset$, put both vertices $i,j$ into $V(F)$ and the edge $ij$ into $E(F)$. It is easy to see that 
    $v(F) \leq 2v(H') = 2v(H)$, 
    $V(H') \subseteq V(A_F)$, and 
    $E(H') \subseteq E(A_F)$.
\end{proof}


If $\girth(G) > 2v(H)$, then the subgraph $F \subseteq G$ given by Lemma \ref{thm:g-locally-a-forest} is a forest. Consequently, we will be able to assume, for most of what follows, that the graph $G$ is itself a forest.
For example, in the cases where $H$ is 3-connected or a cycle, we prove the following:

\begin{theorem}\label{theorem:main-theorem-H-equals-L}
    Suppose that $H$ is $3$-connected or a cycle, and let $s,a,b \in V(H)$ be distinct vertices with $sa,sb \in E(H)$. Let $G$ be a forest with $V(G) = [n]$, let $A = A(G,H,s,a,b)$, and let $C$ be a vertex cover of $G$. Then, 
    $A' := A - \{s'w_i : i \in C\}$ is $H$-free.
\end{theorem}

The proof of Theorem \ref{theorem:main-theorem-H-equals-L} relies on several key properties of the construction $A = A(G,H,s,a,b)$, which we now prove. First we consider certain natural separators of $A$.

\begin{lemma}\label{lem: s',w_t separator}
    Let $A = A(G,H,s,a,b)$ as in Construction \ref{construction:general-construction}. Let 
    $R \subseteq V(G)$ be a separator of $G$, and put $R' := \{s'\} \cup \{w_i : i \in R\}$.
    Then $A - R'$ is disconnected. Moreover, let $C_1,\dots,C_m$ be the connected components of 
    $G - R$, and put
    $$
    D_k := \{w_i : i \in C_k\} \cup 
    \bigcup_{ij \in E(C_k) \cup E(C_k,R)} \left( V(H_{ij}) \setminus \{s',w_i,w_j\} \right)
    $$
    for $k \in [m]$.
    Then there are no paths in $A-R'$ between any two of the sets $D_1,\dots,D_m$.
\end{lemma}
\begin{proof}
    We need to show that for all $1 \leq k < \ell \leq m$, there is no path between $D_k$ and $D_{\ell}$ in $A - R'$. To this end, it suffices to prove the following two claims:
    \begin{enumerate}
        \item $D_k \cap D_{\ell} = \emptyset$ for all $1 \leq k < \ell \leq m$.
        \item For all $1 \leq k \leq m$, all edges leaving $D_k$ go to $R'$.
    \end{enumerate}

    We first prove Item 1. 
    For convenience, let us show (without loss of generality) that $D_1 \cap D_2 = \emptyset$.
    Suppose for contradiction that 
    $v \in D_1 \cap D_2$. If $v = w_i$ for some $i \in V(G)$ then 
    $i \in C_1 \cap C_2$, contradicting that 
    $C_1, C_2$ are disjoint. Thus, $v \in (V(H_{ij}) \setminus \{s', w_i, w_j\}) \cap (V(H_{k\ell}) \setminus \{s', w_k, w_{\ell}\})$ for 
    $ij \in E(C_1) \cup E(C_1,R)$ and 
    $k\ell \in E(C_2) \cup E(C_2,R)$. But then $ij \neq k\ell$ (as $C_1 \cap C_2 = \emptyset$), and therefore
    $$
    (V(H_{ij}) \setminus \{s', w_i, w_j\}) \cap (V(H_{k\ell}) \setminus \{s', w_k, w_{\ell}\}) = \emptyset
    $$
    by Construction \ref{construction:general-construction}, so no such $v$ can exist.
    
    For Item 2, let $uv \in E(A)$ with $u \in D_k$. We need to show that 
    $v \in D_k \cup R'$. If $v = s'$ then this holds, so suppose that $v \neq s'$. Also, 
    $u \neq s'$ because $u \in D_k$. 
    Hence, $uv \notin \{s'w_i : i \in [n]\}$. Then, by Fact \ref{lemma:edges-belong-uniquely}, there exists a unique edge $ij \in E(G)$ such that $uv \in E(H_{ij})$. 
    We now consider two cases. Suppose first that 
    $u \in V(H_{ij}) \setminus \{s',w_i,w_j\}$. Since $u \in D_k$, we then must have 
    $ij \in E(C_k) \cup E(C_k,R)$. As 
    $v \in V(H_{ij})$, it follows that $v \in D_k$ or $v \in R'$, as required. Suppose now that $u \in \{w_i,w_j\}$, say $u = w_i$. Then $i \in C_k$ (because $u \in D_k$). Since $C_k$ is a connected component of $G-R$, we have 
    $j \in C_k \cup R$. Hence, 
    $V(H_{ij}) \setminus \{s',w_i,w_j\} \subseteq D_k$ and $w_j \in D_k \cup R'$. In any case, $v \in D_k \cup R'$. 
\end{proof}

\vspace{0.3cm}

\noindent
Next, we need the following important definition.
\begin{definition}[minimal support]\label{def:minimal edge-set}
    Let $A = A(G,H,s,a,b)$ and let 
    $K$ be a subgraph of $A$. A 
    {\em minimal support} for $K$ is a minimal (with respect to inclusion) subgraph $P$ of $G$ such that
    $K$ is a subgraph of $A_P = A(P,H,s,a,b) \subseteq A$.
\end{definition}
Here, the minimality of $P$ is with respect to the subgraph relation; namely, we assume that removing any vertex or edge from $P$ violates the property that $K \subseteq A_P$. Note that every subgraph $K$ of $A$ admits a minimal support. 
\noindent
The following are easy properties of a minimal support. 
\begin{fact}\label{fact:minimality of P}
    Let $P$ be a minimal support for 
    a subgraph $K \subseteq A$. Then the following hold:
    \begin{enumerate}
        \item For every $ij \in E(P)$, there exists 
        $v \in V(K)$ with 
        $v \in 
        V(H_{ij}) \setminus \{s',w_i,w_j\}$, or there exists 
        $e \in E(K)$ with 
        $e \in E(H_{ij}) \setminus \{s'w_i,s'w_j\}$.
        \item For every $i \in V(P)$, if $i$ is isolated in $P$ then $w_i \in V(K)$ and $d_K(w_i) \leq 1$. Hence, if $\delta(K) \geq 2$ then $P$ has no isolated vertices.
    \end{enumerate}
\end{fact}
\begin{proof}
    For Item 1, take $P' := P - \{ij\}$ (i.e., we delete the edge $ij$ but keep its vertices). By the minimality of $P$, the graph $K$ is not a subgraph of $A_{P'}$. But the only vertices present in $A_P$ and not in $A_{P'}$ are the vertices in $V(H_{ij}) \setminus \{s',w_i,w_j\}$, and the only edges present in $A_P$ and not in $A_{P'}$ are the edges in
    $E(H_{ij}) \setminus \{s'w_i,s'w_j\}$. Item 1 follows.

    For Item 2, suppose that $i \in V(P)$ is isolated in $P$, and take $P' := P - \{i\}$. By the minimality of $P$, the graph $K$ is not a subgraph of $A_{P'}$. But $A_P$ is obtained from $A_{P'}$ by adding the vertex $w_i$ and the edge $s'w_i$. Hence, $w_i \in V(K)$ and $d_K(w_i) \leq d_{A_P}(w_i) = 1$. Item 2 follows. 
\end{proof}

The next simple lemma states that for $A'$ as defined in Theorem \ref{theorem:main-theorem-H-equals-L}, any (hypothetical) copy of $H$ in $A'$ needs to be supported on at least two of the basic copies $H_{ij}$, provided that $\delta(H) \geq 2$. 
\begin{lemma}\label{lem:|P|>=2}
Suppose that $\delta(H) \geq 2$. 
Let $A = A(G,H,s,a,b)$, let $C$ be a vertex-cover of $G$, and let $A' := A - \{s'w_i : i \in C\}$. Let $H'$ be an $H$-copy in $A'$, and let 
$P \subseteq G$ be a minimal support for 
$H'$. Then $e(P) \geq 2$.
Moreover, there is no $ij \in E(G)$ with 
$V(H') \subseteq V(H_{ij})$. 
\end{lemma}
\begin{proof}
We have $\delta(H') = \delta(H) \geq 2$ (by assumption). 
Hence, by Item 2 of Fact \ref{fact:minimality of P}, $P$ has no isolated vertices.
Also, $P$ is not the empty graph, because otherwise $A_P = \{s'\}$ and so $H' \not\subseteq A_P$.
Now suppose by contradiction that $e(P) \leq 1$. 
Since $P$ has no isolated vertices and is not empty, it follows that $P$ is an edge, say $P = \{ij\}$ for some $ij \in E(G)$. 
Then the construction $A_P$ simply equals the basic copy $H_{ij}$, so $H' \subseteq H_{ij}$. 
But
$H'$ and $H_{ij}$ are both copies of $H$, and so $H' = H_{ij}$. 
However, one of the edges $s'w_i,s'w_j \in E(H_{ij})$ was deleted when turning $A$ into $A'$ (because the vertex-cover $C$ contains $i$ or $j$). This contradicts the fact that $H'$ is a subgraph of $A'$.

For the second part, suppose by contradiction that $V(H') \subseteq V(H_{ij})$ for some $ij \in E(G)$. The subgraph of $A$ induced on $V(H_{ij})$ is exactly $H_{ij}$, so again we get that $H' = H_{ij}$, which was already ruled out above. 
\end{proof}

\noindent
The following lemma relates minimal supports of a subgraph $K$ with separators of $K$.

\begin{lemma}\label{claim:separators of H'}
        Let $A = A(G,H,s,a,b)$, 
        let $K$ be a subgraph of $A$, and let 
        $P \subseteq G$ be a minimal support for $K$. 
        Let $R \subseteq V(P)$ be a separator of the graph $P$ and put 
        $R' := \{s'\} \cup \{w_i : i \in R\}$. Then the graph $K - R'$ is disconnected.
        Furthermore, let $C_1,\dots,C_m$ be the connected components of $P - R$, and let 
        \begin{equation}\label{eq:separators of H'}
        D_k := \{ w_i : i \in C_k\} \cup 
        \bigcup_{ij \in E_P(C_k) \cup E_P(C_k,R)}
        \left( V(H_{ij}) \setminus \{s',w_i,w_j\} \right)
        \end{equation}
        for $k=1,\dots,m$. Then $V(K)$ intersects each of the sets $D_1,\dots,D_m$, and there is no path in $K-R'$ between any two of these sets. 
\end{lemma}
    
    \begin{proof}
        By the choice of $P$ (as a minimal support for $K$), we have 
        $K \subseteq A_P = A(P,H,s,a,b)$. 
        By Lemma \ref{lem: s',w_t separator}, applied to $A_P$, we get that $A_P - R'$ is disconnected, and furthermore that there are no paths between any two of the sets $D_1,\dots,D_m$ in $A_P - R'$.
        Thus, it suffices to show that $K$ has vertices in each of the sets $D_1,\dots,D_m$. Without loss of generality, we will prove this for $k=1$.
        Suppose first that there exists $i \in C_1$ which is isolated in $P$. Then by Item 2 of Fact \ref{fact:minimality of P}, $w_i \in V(K)$. Also, by definition, $w_i \in D_1$, so we are done. 
        From now on, suppose that no vertex in $C_1$ is isolated in $P$.
        Then there must exist an edge in 
        $E_P(C_1) \cup E_P(C_1,R)$.
        Let $ij$ be such an edge with 
        $i \in C_1$.
        By Item 1 of Fact \ref{fact:minimality of P}, $K$ either contains a vertex 
        $v \in V(H_{ij}) \setminus \{s',w_i,w_j\}$ or an edge
        $e \in E(H_{ij}) \setminus \{s'w_i,s'w_j\}$. 
        In the former case, $v \in D_1$ and we are done. So we may assume that 
        $K$ has no vertices in 
        $V(H_{ij}) \setminus \{s',w_i,w_j\}$. Then the only option for the edge $e$ (in the latter case) is $e = w_iw_j$. But then $w_i \in V(K)$ (because $e \in E(K)$) and also $w_i \in D_1$ (because $i \in C_1$). Thus, $V(K) \cap D_1 \neq \emptyset$, as required. 
    \end{proof}

\noindent
We can now prove Theorem \ref{theorem:main-theorem-H-equals-L}.

\begin{proof}[Proof of Theorem \ref{theorem:main-theorem-H-equals-L}]
Suppose by contradiction that 
$A' = A - \{s'w_i : i \in C\}$ contains a copy $H'$ of $H$. Let $P \subseteq G$ be a minimal support for $H'$ (as in Definition \ref{def:minimal edge-set}). 
Clearly, $\delta(H') = \delta(H) \geq 2$.
Hence, by Lemma \ref{lem:|P|>=2}, $e(P) \geq 2$.
Since $P$ is a forest (as a subgraph of $G$), it has a separator $R$ of size at most 1 (this is true of any forest with at least 2 edges). 
Put $R' := \{s'\} \cup \{w_i : i \in R\}$. 
By Lemma \ref{claim:separators of H'}, applied with 
$K := H'$, we get that $H' - R'$ is disconnected. 
If $H$ is 3-connected then this immediately gives a contradiction, because $H'$ is 3-connected (as a copy of $H$) and $|R'| \leq 2$. 
From now on, suppose that $H$ is a cycle.
Note that $R' \subseteq V(H')$, because otherwise $H'$ is not 2-connected (which contradicts $H$ being a cycle). In particular, this means that $s' \in V(H')$.  

Next, we claim that 
$V(H') \not 
\subseteq 
\{s',w_1,\dots,w_n\}$. 
Indeed, suppose otherwise. 
If $H$ is a triangle, then either 
$V(H') = \{w_i, w_j, w_k\}$ for some $ij, jk, ki \in E(G)$ (in which case $G$ contains a triangle, contradicting that $G$ is a forest) or 
$V(H') = \{s', w_i, w_j\}$ for some $ij \in E(G)$ (which contradicts that either $s'w_i \not \in E(A')$ or $s'w_j \not \in E(A')$). And if $H$ is a cycle of length at least $4$, then
$w_iw_j \not \in E(A)$ for all $ij \in E(G)$, and thus $A[\{s',w_1,\dots,w_n\}]$ is a star and hence contains no cycle. So we conclude that there exists $ij \in E(G)$ such that 
\begin{equation}\label{eq:H = L proof}
V(H') \cap (V(H_{ij}) \setminus \{s',w_i,w_j\}) \neq \emptyset.
\end{equation}

By Fact \ref{fact:outside vertices}, $s',w_i,w_j$ are the only vertices in $V(H_{ij})$ which may have neighbors outside of $V(H_{ij})$. Also,
every vertex $v \in V(H_{ij}) \setminus \{s',w_i,w_j\}$ has $d_A(v) = 2$ (because $H$ is a cycle). Hence, if $v \in V(H_{ij}) \setminus \{s',w_i,w_j\}$ belongs to $V(H')$, then both neighbors of $v$ in $H_{ij}$ also belong to $V(H')$ (since $H'$ is 2-regular). By applying this repeatedly and using \eqref{eq:H = L proof}, we get that $V(H_{ij}) \setminus \{s'\} \subseteq V(H')$. We also already saw that $s' \in V(H')$. Finally, since $V(H') \not \subseteq V(H_{ij})$ (by the second part of Lemma \ref{lem:|P|>=2}), we see that $|V(H')| > |V(H_{ij})|$, which is impossible as $H',H_{ij}$ are both copies of $H$. This completes the proof. 
\end{proof}

\noindent
We end this section by proving a few other basic properties of Construction \ref{construction:general-construction}.
For a path $P = (v_1,\dots,v_m)$ and indices $1 \leq i \leq j \leq m$, we denote by $P[v_i,v_j]$ the segment of $P$ between $v_i$ and $v_j$. 

\begin{lemma}\label{claim:exists-subpath-between-vertices-in-same-copy}
        Let $A = A(G,H,s,a,b)$ and suppose that $G$ is a forest. Let $ij \in E(G)$, let
        $x, y \in V(H_{ij})$, and let $Q$ be an $x,y$-path in $A$. Then there exists an $x,y$-path $Q'$ such that 
        $V(Q') \subseteq V(H_{ij}) \cap V(Q)$.
\end{lemma}
    \begin{proof}
        Let $Q' := (x = v_1, \dots, v_m = y)$ be a minimum-length $x,y$-path with $V(Q') \subseteq V(Q)$. If $V(Q') \subseteq V(H_{ij})$ then we are done, so
        suppose by contradiction that $V(Q') \not \subseteq V(H_{ij})$. Hence, there is some minimal index $k + 1$ such that $v_{k + 1} \not \in V(H_{ij})$ (i.e., $v_k \in V(H_{ij})$), and there is some minimal index $\ell > k + 1$ such that $v_{\ell} \in V(H_{ij})$. 
        The choice of $k,\ell$ guarantees that $v_{k+1},v_{k+2},\dots,v_{\ell-1} \notin V(H_{ij})$. 
        
        By Fact \ref{fact:outside vertices}, $s'$, $w_i$ and $w_j$ are the only vertices in $H_{ij}$ which may have neighbors outside of $H_{ij}$.
        Hence,
        $v_k,v_\ell \in \{s',w_i,w_j\}$.
        We claim that $s' \in \{v_k, v_{\ell}\}$. Indeed, suppose by contradiction that $s' \neq v_k, v_{\ell}$. Then, without loss of generality, we may assume that $v_k = w_i$ and $v_{\ell} = w_j$.
        By Fact \ref{lemma:edges-belong-uniquely}, there exists
        $e \in E(G)$ such that 
        $v_{\ell-1}v_{\ell} \in E(H_e)$. 
        Then $e \neq ij$ because $v_{\ell-1} \notin V(H_{ij})$ (by the minimality of $\ell$). Also, since $w_j = v_{\ell} \in V(H_e)$, we have that $j \in e$. Thus, we can write $e = i'j$ for some $i' \neq i$. Since $ij,i'j \in E(G)$ and $G$ is a forest, the vertices $i,i'$ are in different connected components of 
        $G-\{j\}$. Hence, by Lemma 
        \ref{lem: s',w_t separator} with $R = \{j\}$, we get that 
        $A - \{s',w_j\}$ has no paths between
        $X := V(H_{ij}) \setminus \{s',w_j\}$ and 
        $Y := V(H_{i'j}) \setminus \{s',w_j\}$. But $v_k = w_i \in X$ and $v_{\ell-1} \in Y$, and $Q'[v_k,v_{\ell-1}]$ is a 
        $v_k,v_{\ell-1}$-path not containing $s',w_j$ (because $v_{k+1},\dots,v_{\ell-1} \notin V(H_{ij})$ and $v_k = w_i \neq s',w_j$). This is a contradiction, proving our claim that 
        $s' \in \{v_k,v_\ell\}$. 
        
       Without loss of generality, suppose that $v_k = s'$ (the case that $v_\ell = s'$ is handled the same way). Then $v_\ell \in \{w_i,w_j\}$. In any case, $v_kv_\ell \in E(A)$ (since $s'w_i,s'w_j \in E(A)$). But then we can shorten $Q'$ by using the edge $v_kv_{\ell}$ in place of the path $Q'[v_k,v_{\ell}]$. 
       (This is indeed a shortening because $\ell-k \geq 2$.)
       This contradicts the minimality of $Q'$ and completes the proof. 
    \end{proof}

\begin{lemma}\label{lemma:intersections with H_{ij} are k-connected}
    Let $A = A(G,H,s,a,b)$ and suppose that $G$ is a forest. Let $k \geq 1$, let $X \subseteq V(A)$, and suppose that $A[X]$ is $k$-connected. Then for every $ij \in E(G)$, $A[X \cap V(H_{ij})]$ is $k$-connected or 
    $|X \cap V(H_{ij})| \leq k$.
\end{lemma}
\begin{proof}
    Suppose that $|X \cap V(H_{ij})| \geq k+1$. Our goal is to show that for every pair of non-adjacent distinct vertices $x,y \in X \cap V(H_{ij})$, there exist $k$ internally vertex-disjoint $x,y$-paths inside $A[X \cap \nolinebreak V(H_{ij})]$. 
    Let $x,y \in X \cap V(H_{ij})$ be non-adjacent. 
    Since $A[X]$ is $k$-connected, there are $k$ internally vertex-disjoint $x,y$-paths $Q_1,\dots,Q_k$ in $A[X]$. By Lemma \ref{claim:exists-subpath-between-vertices-in-same-copy}, for each $i \in [k]$ there exists an $x,y$-path $Q_i'$ with $V(Q'_i) \subseteq V(H_{ij})$ and $V(Q'_i) \subseteq V(Q_i)$.
    Then $Q_1',\dots,Q_k'$ are internally vertex-disjoint and contained in $X \cap V(H_{ij})$. 
    Note also that $Q'_1,\dots,Q'_k$ are pairwise-distinct because $x,y$ are non-adjacent. 
    This completes the proof. 
\end{proof}

\begin{lemma}\label{lemma:H'-contains-intermediate-vertices-generalized}
        Let $A = A(G,H,s,a,b)$ and suppose that $G$ is a forest. 
        Let $k \geq 1$ and suppose that $H$ is $k$-connected.
        Let $S$ be a separator of $H$ with $|S| = k$, and let $p,q \in V(H) \setminus S$ be two vertices separated by $S$. 
        Let $H'$ be a copy of $H$ in $A$.
        Let $ij \in E(G)$ and suppose that
        $f_{ij}(p), f_{ij}(q) \in V(H')$. Then 
        $f_{ij}(S) \subseteq V(H')$, and 
        $f_{ij}(S)$ is a separator of $H'$. 
\end{lemma}
\begin{proof}
    It suffices to show that every 
    $f_{ij}(p), f_{ij}(q)$-path in $H'$ contains a vertex from $f_{ij}(S)$. Indeed, since $H'$ is $k$-connected (as a copy of $H$), it would follow that 
    $f_{ij}(S) \subseteq V(H')$ and that 
    $f_{ij}(S)$ is a separator of $H'$. So let $Q$ be an $f_{ij}(p), f_{ij}(q)$-path in $H'$. By Lemma \ref{claim:exists-subpath-between-vertices-in-same-copy}, there is an 
    $f_{ij}(p), f_{ij}(q)$-path $Q'$ with 
    $V(Q') \subseteq V(H_{ij}) \cap V(Q)$. Now, $f_{ij}^{-1}(Q')$ is a $p,q$-path in $H$ (because $f_{ij} : H \rightarrow H_{ij}$ is an isomorphism). Since $S$ separates $p,q$, we have 
    $f_{ij}^{-1}(V(Q')) \cap S \neq \emptyset$, and hence $V(Q') \cap f_{ij}(S) \neq \emptyset$. This shows that $Q$ contains a vertex of $f_{ij}(S)$, as required.  
\end{proof}

\noindent

\section{2-connected $H$}\label{sec:2-connected}
The main result of this section is the extension of Theorem \ref{theorem:main-theorem-H-equals-L} to 2-connected graphs $H$, as follows:
\begin{theorem}\label{theorem:2-connected H}
    For every 2-connected graph $H$, there exist distinct vertices $s,a,b \in V(H)$ with $sa,sb \in E(H)$ such that the following holds. Let $G$ be a forest with $V(G) = [n]$ and let 
    $A = A(G, H, s, a, b)$ as in Construction \ref{construction:general-construction}. Let $C$ be a vertex-cover of $G$. Then 
    $A' := A - \{s'w_i : \nolinebreak i \in \nolinebreak C\}$ is $H$-free. 
    Furthermore, given any $z \in V(H)$, the three vertices $s,a,b$ can be chosen such that $s \neq z$.
\end{theorem}
\noindent
The ``Furthermore" part of Theorem \ref{theorem:2-connected H} will be needed in Section \ref{sec:connected, delta >=2}.
Towards Theorem \ref{theorem:2-connected H}, we first establish some structural properties of $H$ in Section \ref{subsec:structure of H}. Theorem \ref{theorem:2-connected main} is then proved in Section \ref{subsec:2-connected main}.

\subsection{The structure of $H$}\label{subsec:structure of H}

The first main idea towards proving Theorem \ref{theorem:2-connected H} is to use Tutte's \cite{Tutte} decomposition of 2-connected graphs (see also \cite{Carmesin,CunninghamEdmonds}). This result states that every 2-connected graph has a tree-decomposition of adhesion 2 where every torso is 3-connected or a cycle. We now define these terms. 

\begin{definition}\label{def:tree decomposition}
    A {\em tree-decomposition} of a graph $H$ is a pair $(T,(B_t)_{t \in V(T)})$, where $T$ is a tree and $B_t \subseteq V(H)$ for all $t \in V(T)$, satisfying the following axioms:
    \begin{enumerate}
        \item $\bigcup_{t \in V(T)}B_t = V(H)$, and every edge of $H$ is contained in $B_t$ for some $t \in V(T)$.
        \item For every $v \in V(H)$, the set 
        $\{t \in V(T) : v \in B_t\}$ induces a connected subgraph (i.e., a subtree) of $T$.
    \end{enumerate}
\end{definition}

\noindent
The sets $B_t$ are called {\em bags}.
We also need the following notions. Let $\mathcal{D} = (T,(B_t)_{t \in V(T)})$ be a tree-decomposition.
\begin{itemize}
        \item $\mathcal{D}$ has {\em adhesion $2$} if $|B_t \cap B_{t'}| = 2$ for all $tt' \in E(T)$. 
        \item Assuming that $\mathcal{D}$ has adhesion 2, a {\em torso} of $\mathcal{D}$ is any graph obtained from the graph $H[B_t]$ (for a bag $B_t$) by adding an edge on the pair $B_t \cap B_{t'}$ for every $t' \in N_T(t)$ (if this edge is not already present).  
\end{itemize}


\noindent
The following is a basic and well-known property of tree-decompositions:
\begin{lemma}\label{lem:tree decomposition separator}
Let $(T,(B_t)_{t \in V(T)})$ be a 
tree-decomposition of a graph $H$. Let $tt' \in E(T)$, and let $C,C'$ be the two connected components of $T - \{tt'\}$. Then $B_t \cap B_{t'}$ separates $V := \bigcup_{s \in C} B_s$ and $V' := \bigcup_{s \in C'}B_s$ in $H$.
\end{lemma}
\begin{proof}
    First we observe that 
    $V \cap V' = B_t \cap B_{t'}$. 
    Clearly, $B_t \cap B_{t'} \subseteq V \cap V'$.
    Conversely, let $x \in V \cap V'$ and let $s \in C, s' \in C'$ such that $x \in B_s,B_{s'}$. Then $t,t'$ are on the unique $s,s'$-path in $T$. By Item 2 in Definition \ref{def:tree decomposition}, this means that $x \in B_t,B_{t'}$, as required.

    Next, let $xx' \in E(H)$ with $x \in V$ and $x' \in V'$. By Item 1 of Definition \ref{def:tree decomposition}, there exists $s \in V(T)$ with $x,x' \in B_s$. If $s \in C$ then $x' \in V \cap V' = B_t \cap B_{t'}$, and if $s \in C'$ then $x \in V \cap V' = B_t \cap B_{t'}$. In any case, every edge between $V$ and $V'$ intersects $B_t \cap B_{t'}$. 
\end{proof}

The following is Tutte's decomposition theorem for 2-connected graphs, stated in the modern language of tree-decompositions. This formulation appears in \cite{Carmesin}. The original theorem appears in \cite{CunninghamEdmonds,Tutte}.

\begin{theorem}\label{thm:Tutte}
Every 2-connected graph $H$ has a tree-decomposition of adhesion 2 such that each of its torsos is 3-connected or a cycle. 
\end{theorem}

Throughout the rest of this section, we assume that $H$ is a $2$-connected graph which is not $3$-connected or a cycle. 
We will apply Theorem \ref{thm:Tutte} and consider bags corresponding to leaves of the decomposition tree $T$.
We will (implicitly) see that if the torso of such a bag is a cycle, then one can refine the tree-decomposition such that some leaf torso is a triangle. Restricting such ``leaf cycles" to be triangles will simplify our analysis.
This leads to the following definition:
\begin{definition}\label{def:leaf}
    Let $H$ be a 2-connected graph. A 
    {\em leaf} in $H$ is a subset $L \subseteq V(H)$ satisfying one of the following: 
    \begin{itemize}
        \item $L = V(H)$ and $H$ is 3-connected or a cycle.
        \item $V(H) \setminus L \neq \emptyset$, and there exist $x, y \in L$ such that 
        $\{x, y\}$ separates $L \backslash \{x, y\}$ from $V(H) \backslash L$. Furthermore, 
        let $\overline{L}$ be the graph with vertex-set $L$ and edge-set 
        $E(L) \cup \{xy\}$. Then, $\overline{L}$ is 3-connected or a triangle. We call $\{x, y\}$ the {\em separator} of $L$.
    \end{itemize}
\end{definition}
\noindent
With a slight abuse of notation, we denote by $L$ both the vertex-set and the induced subgraph $H[L]$.

We will obtain leaves (as per Definition \ref{def:leaf}) by applying Theorem \ref{thm:Tutte} and taking bags which correspond to leaves of the tree $T$. We note that this is not a one-to-one correspondence: a leaf $L$ as per Definition \ref{def:leaf} need not come from a leaf-bag.
The advantage of Definition \ref{def:leaf} is that if $L$ is a leaf of $H$ (as per Definition \ref{def:leaf}), then $L$ is also a leaf of every 2-connected induced subgraph $F$ of $H$ with $L \subsetneq V(F)$. This will be important later on. Another advantage is that this definition is independent from the tree-decomposition.

\begin{proposition}\label{prop:2-connected leaves}
    Let $H$ be a 2-connected graph, and suppose that $H$ is neither 3-connected nor a cycle. Then $H$ has at least two leaves $L_1,L_2$. Furthermore, if $L_i$ has separator $\{x_i,y_i\}$, $i=1,2$, then 
    \begin{equation}\label{eq:distinct leaves are internally disjoint}
    (L_1 \setminus \{x_1, y_1\}) \cap (L_2 \setminus \{x_2, y_2\}) = \emptyset.
    \end{equation}
\end{proposition}
\begin{proof}
    Take a tree-decomposition 
    $\mathcal{D} = (T,(B_t)_{t \in V(T)})$ of $H$ as in Theorem \ref{thm:Tutte}. If $|V(T)| = 1$, then $H$ is 3-connected or a cycle, contradicting the assumption of the proposition. Hence, $|V(T)| \geq 2$. Every tree on at least 2 vertices has at least 2 leaves.
    Consider any leaf $\ell \in V(T)$ of $T$, and take $L := B_{\ell}$. 
    Let $t$ be the unique neighbor of $\ell$ in $T$. 
    Since $\mathcal{D}$ has adhesion 2, we can write $B_{\ell} \cap B_t = \{x,y\}$. 
    By Lemma \ref{lem:tree decomposition separator},  
    $\{x,y\}$ separates 
    $B_{\ell} = L$ from 
    $\bigcup_{s \in V(T) \setminus \{\ell\}} B_s$. In particular, this means that 
    $L \setminus \{x,y\}$ is disjoint from $\bigcup_{s \in V(T) \setminus \{\ell\}} B_s$ (because otherwise there would be a trivial path between these two sets avoiding $x,y$). 
    Hence, 
    $\bigcup_{s \in V(T) \setminus \{\ell\}} B_s = (V(H) \setminus L) \cup \{x,y\}$, and $\{x,y\}$ separates $L \setminus \{x,y\}$ from $V(H) \setminus L$. 
    Also, since $|V(T)| \geq 2$, we have $V(H) \setminus L \neq \emptyset$; because otherwise, i.e., if $L = V(H)$, then the intersection of $L$ and each of its neighboring bags $B$ would have size $|B| \geq 3$, using that all torsos are cycles or 3-connected. But this would contradict that $\mathcal{D}$ has adhesion 2.
    
    Next, Theorem \ref{thm:Tutte} guarantees that $\overline{L} := L + \{xy\}$ is 3-connected or a cycle (because $\overline{L}$ is the torso corresponding to $B_\ell = L$). If $\overline{L}$ is 3-connected then we are done. Else (i.e., if $\overline{L}$ is a cycle) take three vertices $s,a,b \in L$ with $sa,sb \in E(H)$ and $s \notin \{x,y\}$, and replace $L$ with $L' := \{s,a,b\}$. It is easy to see that $L'$ is a leaf as per Definition \ref{def:leaf}, with separator $\{a,b\}$. 

    The above shows that $H$ has (at least) two leaves $L_1,L_2$, 
    and that each $L_i$ is contained in  $B_{\ell_i}$ for some distinct leaves $\ell_1,\ell_2$ of $T$. Let 
    $\{x_i,y_i\}$ be the separator of $L_i$. We saw above that $L_i \setminus \{x_i,y_i\}$ is disjoint from 
    $\bigcup_{s \in V(T) \setminus \{\ell_i\}} B_s$, and hence in particular disjoint from $L_{3-i}$. This proves \eqref{eq:distinct leaves are internally disjoint}.
\end{proof}

\noindent
Next, we collect some simple facts about leaves of 2-connected graphs.

\begin{lemma}\label{claim:H-L-x-y-connected}
    Let $H$ be a $2$-connected graph which is not $3$-connected or a cycle, and let $L \subseteq V(H)$ be a leaf of $H$ with separator $\{x, y\}$, as in Definition \ref{def:leaf}. 
    Then: 
    \begin{enumerate}
        \item 
        Let $H_1$ be the subgraph of $H$ induced by 
        $V(H) \setminus (L \setminus \{x,y\})$.
        Then $H_1 + \{xy\}$ is 2-connected, and hence 
        $H_1$ is connected. 
        \item For every $u \in V(H) \setminus L$, there is a $u,y$-path which avoids $L \setminus \{y\}$, as well as a $u,x$-path which avoids $L \setminus \{x\}$. 
    \end{enumerate}
\end{lemma}

\begin{proof}
    We start with Item 1. 
    We need to show that $H_1 + \{xy\}$ has two internally vertex-disjoint $u,v$-paths for every pair of non-adjacent distinct vertices $u,v \in V(H_1)$. 
    In particular, 
    $\{u,v\} \neq \{x,y\}$. 
    Since $H$ is 2-connected, there are two internally vertex-disjoint $u,v$-paths $P,Q$ in $H$. If one of these paths, say $P$, has vertices in 
    $L \setminus \{x,y\}$, then $P$ must contain both $x$ and $y$, since $\{x,y\}$ separates 
    $L \setminus \{x,y\}$ from $V(H) \setminus L$. We can then shorten $P$ by replacing the segment of $P$ between $x$ and $y$ with the edge $xy$. This gives two internally vertex-disjoint $u,v$-paths in $H_1 + \{xy\}$, as required.

    Item 2 follows from Item 1. Indeed, by Item 1, there are two internally vertex-disjoint $u,y$-paths in $H_1 + \{xy\}$. Only one of these paths can contain the vertex $x$. The other path avoids $x$ (and hence the edge $xy$), giving a $u,y$-path which avoids $L \setminus \{y\}$. The same argument gives a $u,x$-path which avoids 
    $L \setminus \{x\}$. 
\end{proof}

\begin{lemma}\label{claim:three-VDPs}
    Let $H$ be a $2$-connected graph which is not $3$-connected or a cycle, and let $L \subseteq V(H)$ be a leaf of $H$. If $\overline{L}$ is 3-connected then for all distinct vertices $u, v \in L$, there exist three internally vertex-disjoint $u,v$-paths in $H$.
\end{lemma}

\begin{proof}
    By assumption, there exist three internally vertex-disjoint $u,v$-paths $P_1, P_2, P_3$ in $\overline{L}$. Note that at most one such path uses the edge $xy$.
    If none of the paths $P_1,P_2,P_3$ uses $xy$ then these paths lie in $L$ (and thus, in $H$), and we are done. So suppose, without loss of generality, that $P_3$ uses $xy$. 
    Also, we may assume that $xy \notin E(H)$ (else we are done).
    By Item 1 of Lemma \ref{claim:H-L-x-y-connected}, $H_1 + \{xy\}$ is 2-connected, and hence $H_1$ is connected. 
    Let $Q$ be an $x,y$-path in $H_1$. In particular, $Q$ does not use the edge $xy$ (since $xy \notin E(H)$).
    By replacing the edge $xy$ in $P_3$ with the path $Q$, we get a $u,v$-path $P'_3$. Note that $P'_3$ is internally vertex-disjoint from $P_1,P_2$, because $P_3$ is internally vertex-disjoint from $P_1,P_2$ and $P_1,P_2 \subseteq L$. So $P_1,P_2,P'_3$ give the desired result.  
\end{proof}


\begin{lemma}\label{lemma:sv-not-2-cut-if-H-is-3-connected}
    Let $H$ be a $2$-connected graph which is not $3$-connected or a cycle. Let $L \subseteq V(H)$ be a leaf of $H$ with separator $\{x,y\}$, and suppose that $\overline{L}$ is 3-connected.
    Let $s \in L \setminus \{x,y\}$ and let $v \in V(H)$. Then, $\{s, v\}$ is not a separator of $H$.
\end{lemma}

\begin{proof}

    We prove that $H - \{s,v\}$ is connected in each of the following cases. As in Lemma \ref{claim:H-L-x-y-connected}, let $H_1$ be the subgraph of $H$ induced by 
    $V(H) \setminus (L \setminus \{x,y\})$.
    
    \textbf{Case 1: $v \in L \backslash \{x, y\}$.} 
    By Item 1 of Lemma \ref{claim:H-L-x-y-connected}, 
    $H_1$ is connected. 
    Also, for each $u \in L \setminus \{s,v,y\}$, there are three internally vertex-disjoint $u,y$-paths in $H$ (by Lemma \ref{claim:three-VDPs}). At least one of these paths avoids the set $\{s,v\}$. So we see that every $u \in V(H) \setminus \{s,v\}$ is connected by a path to $y$ in $H - \{s,v\}$. Hence, $H - \{s,v\}$ is connected.
    
    \textbf{Case 2: $v = x$ or $v = y$.} Without loss of generality, suppose that $v = x$. We claim that every vertex $u \in V(H) \setminus \{s, v\}$ is connected via a path to $y$ in $H - \{s,v\}$. This is trivial for $u=y$, and if $u \in V(H) \setminus L$ then this holds by Item 2 of Lemma \ref{claim:H-L-x-y-connected}. Indeed, this lemma gives a $u,y$-path in $H$ which avoids $L \setminus \{y\}$ and hence avoids $\{s,x\} = \{s,v\}$. Suppose now that $u \in L \setminus \{y\}$. As in Case 1, Lemma \ref{claim:three-VDPs} gives three internally vertex-disjoint $u,y$-paths in $H$, one of which avoids \nolinebreak  $\{s,x\}$, \nolinebreak as \nolinebreak required. 
    
    
    \textbf{Case 3: $v \in V(H) \backslash L$.} 
    First, we claim that every vertex 
    $w \in L \backslash \{s, x, y\}$ is connected to both $x$ and $y$ in $L \backslash \{s\}$. Indeed, as $\overline{L}$ is 3-connected, there exist three internally vertex-disjoint $w,x$-paths in $\overline{L}$. At most one such path can use the edge $xy$, and at most one such path can use the vertex $s$. Thus, there exists at least one $w,x$-path in $L \backslash \{s\}$, as desired. An identical argument gives a $w,y$-path in $L \backslash \{s\}$. 
    Note that $|L| \geq 4$ (since $\overline{L}$ is 3-connected) and so 
    $L \backslash \{s, x, y\} \neq \emptyset$.
    Hence, we showed that the vertices of 
    $L \setminus \{s\}$ are all in the same connected component of $H - \{s,v\}$.

    In order to show that $H - \{s,v\}$ is connected, it now suffices to show that every
    $w \in V(H) \backslash (L \cup \{v\})$ is connected by a path in $H - \{s, v\}$ to $x$ or $y$. 
    And indeed, by Item 1 of Lemma \ref{claim:H-L-x-y-connected}, $H_1 + \{xy\}$ is 2-connected, and hence $(H_1 + \{xy\}) - \{v\}$ is connected. This means that $H_1 - \{v\}$ has at most two connected components, and each of these components contains $x$ or $y$. This proves our claim.
\end{proof}

\begin{lemma}\label{claim:sa-sb-not-2-cuts}
    Let $H$ be a $2$-connected graph which is not $3$-connected or a cycle, and let $L \subseteq V(H)$ be a leaf of $H$ with separator $\{x, y\}$.
    Let $s \in L \setminus \{x,y\}$ and let $a \in L$ be a neighbor of $s$.
    Then $\{s, a\}$ is not a separator of $H$.
\end{lemma}

\begin{proof}
    If $\overline{L}$ is 3-connected then we are done by Lemma 
    \ref{lemma:sv-not-2-cut-if-H-is-3-connected}. So suppose that $\overline{L}$ is a triangle. In particular, $a \in \{x,y\}$; say $a=x$. 
    Let $H_1$ be the subgraph of $H$ induced by 
    $V(H) \setminus (L \setminus \{x,y\})$. By Item 1 of Lemma \ref{claim:H-L-x-y-connected}, 
    $H_1 + \{xy\}$ is 2-connected. 
    But $H - \{s,x\} = (H_1 + \{xy\}) - \{x\}$, so 
    $H - \{s,x\}$ is connected, as required. 
\end{proof}

\subsection{Proof of Theorem \ref{theorem:2-connected H}}\label{subsec:2-connected main}

\noindent
We will prove Theorem \ref{theorem:2-connected H} in the following form:
\begin{theorem}\label{theorem:2-connected main}
    Let $H$ be a 2-connected graph. 
    Let $s,a,b \in V(H)$ be distinct vertices such that $sa,sb \in E(H)$, and moreover, if $H$ is neither 3-connected nor a cycle then there exists a leaf $L$ of $H$ with separator $\{x,y\}$ such that $s,a,b \in L$ and
    $s \notin \{x,y\}$. Let $G$ be a forest with $V(G) = [n]$ and let $A = A(G,H,s,a,b)$. Let $C$ be a vertex cover of $G$. Then $A' := A - \{s'w_i : i\in C\}$ is $H$-free. 
\end{theorem}

\noindent
First, let us observe that Theorem \ref{theorem:2-connected main} implies Theorem \ref{theorem:2-connected H}.
\begin{proof}[Proof of Theorem \ref{theorem:2-connected H}]
We only need to show that there is a choice of 
$s,a,b \in V(H)$ such that the conditions of Theorem \ref{theorem:2-connected main} are satisfied and $s \neq z$. If $H$ is 3-connected or a cycle then simply take an arbitrary $s \in V(H) \setminus \{z\}$ and take $a,b$ to be two distinct neighbors of $s$. Suppose now that $H$ is neither 3-connected nor a cycle. By Proposition \ref{prop:2-connected leaves}, $H$ has two leaves $L_1,L_2$. Let $\{x_i,y_i\}$ denote the separator of $L_i$, $i=1,2$. Since the sets 
$L_1 \setminus \{x_1,y_1\}, L_2 \setminus \{x_2,y_2\}$ are disjoint (by Proposition \ref{prop:2-connected leaves}), there exists $i \in \{1,2\}$ such that 
$z \notin L_i \setminus \{x_i,y_i\}$. Now pick $s$ to be an arbitrary element of $L_i \setminus \{x_i,y_i\}$, and pick $a,b$ to be two distinct neighbors of $s$ in $L_i$; this is possible since $\overline{L_i}$ is 3-connected or a triangle.
\end{proof}

\begin{proof}[Proof of Theorem \ref{theorem:2-connected main}]
    The proof is by induction on $|V(H)|$. The base cases are when $H$ is 3-connected or a cycle, and these are handled by Theorem \ref{theorem:main-theorem-H-equals-L}. Hence, we assume from now on that $H$ is neither 3-connected nor a cycle. Thus, by the choice of $s,a,b$, there exists a leaf $L$ of $H$ with separator $\{x,y\}$ such that $s,a,b \in L$ and $s \notin \{x,y\}$. Note that if $\overline{L}$ is a triangle then 
    $\{x,y\} = \{a,b\}$.

    Suppose for the sake of contradiction that 
    $A' = A - \{s'w_i : i \in C\}$ contains a copy of $H$, denoted $H'$. 
    Let $P \subseteq G$ be a minimal support for $H'$, as in Definition \ref{def:minimal edge-set}. 
    By Lemma \ref{lem:|P|>=2} we have 
    $e(P) \geq 2$, and by Fact \ref{fact:minimality of P}, $P$ has no isolated vertices (here we use that $\delta(H') = \delta(H) \geq 2$, as $H$ is 2-connected).
    Also, $P$ is a forest (as a subgraph of $G$). 
    We now establish some basic properties of $H'$ and $P$.

    \begin{claim}\label{claim:simple properties of H',P}
        {\color{white} text}
        \begin{enumerate}
            \item $P$ is connected.
            \item $s' \in V(H')$.
            \item For every $i \in V(P)$ with 
            $d_P(i) \geq 2$, it holds that $w_i \in V(H')$ and $\{s',w_i\}$ is a separator of $H'$.
            \item For every $ij \in E(P)$, it holds that $w_i \in V(H')$ or $w_j \in V(H')$.
            \item If $\overline{L}$ is a triangle then $w_i,w_j \in V(H')$ for every $ij \in E(P)$.
            \item For every $ij \in E(P)$, the graph $A[V(H') \cap V(H_{ij})]$ is 2-connected.
        \end{enumerate}
    \end{claim}
    \begin{proof}
        We will apply Lemma \ref{claim:separators of H'} to the graph $K := H'$. 
        If $P$ is not connected then by Lemma \ref{claim:separators of H'} with $R := \emptyset$, the graph $H' - \{s'\}$ is disconnected, contradicting that $H'$ is 2-connected (since $H$ is 2-connected and $H'$ is a copy of $H$). This proves Item 1.

        For Items 2-3, consider any 
        $i \in V(P)$ with $d_P(i) \geq 2$. Then $R := \{i\}$ is a separator of $P$ (since $P$ is a tree). Now, again by Lemma \ref{claim:separators of H'}, we see that 
        $H' - \{s',w_i\}$ is disconnected. Since 
        $H'$ is 2-connected, it follows that $s',w_i \in V(H')$. 
        This proves Item 3. To infer Item 2, it suffices to note that $P$ has at least one vertex $i$ with $d_P(i) \geq 2$ (meaning that the above implication is not vacuous). Such an $i$ exists because $P$ is a tree with at least 2 edges. So we conclude that $s' \in V(H')$, as required.
        Note also that if $ij \in E(P)$ then 
        $d_P(i) \geq 2$ or $d_P(j) \geq 2$, so Item 4 follows from Item 3.

        Next, we prove Item 6. So let $ij \in E(P)$. First we claim that 
        $|V(H') \cap V(H_{ij})| \geq 3$. 
        Indeed, as we saw above, we have $d_P(i) \geq 2$ or $d_P(j) \geq 2$; without loss of generality, $d_P(j) \geq 2$. Then $s',w_j \in V(H')$ by Items 2,3. By Fact \ref{fact:minimality of P}, $H'$ either contains a vertex in $V(H_{ij}) \setminus \{s',w_i,w_j\}$ or an edge in
        $E(H_{ij}) \setminus \{ s'w_i,s'w_j \}$. In either case, this gives a vertex in $V(H') \cap V(H_{ij})$ not belonging to 
        $\{s',w_j\}$, proving our claim that $|V(H') \cap V(H_{ij})| \geq 3$. 
        We now apply Lemma \ref{lemma:intersections with H_{ij} are k-connected} with $X := V(H')$ and $k := 2$. Note that $A[V(H')]$ is 2-connected because $H'$ is 2-connected and $H'$ is a spanning subgraph of $A[V(H')]$.
        So by Lemma \ref{lemma:intersections with H_{ij} are k-connected},
        $A[V(H') \cap V(H_{ij})]$ is 2-connected.

        Finally, we prove Item 5. By Item 6, $A[V(H') \cap V(H_{ij})]$ is 2-connected, and hence has minimum degree at least 2. 
        By Item 2, $s' \in V(H')$.
        But note that if $\overline{L}$ is a triangle then the only neighbors of $s'$ in $H_{ij}$ are $w_i,w_j$. Therefore, 
        $w_i,w_j \in V(H')$.
    \end{proof}


    

    With the above properties in hand, we now arrive at the main part of the proof of the theorem. 
    Recall that $f_{ij} : H \rightarrow H_{ij}$ is the natural isomorphism from $H$ to $H_{ij}$.
    We define $F \subseteq H$ to be the subgraph of $H$ induced on the vertex set
    $$V(F) := L \cup 
    \Big(\bigcup_{ij \in E(P)}f_{ij}^{-1}\left( 
    V(H') \cap V(H_{ij}) \right)
    \Big).$$
    That is, $F$ is the induced subgraph of $H$ on the vertices in $L$ alongside the inverse image of each vertex used in the copy $H'$ of $H$. 
    
    In what follows, we will apply the induction hypothesis to the graph $F$. To this end, we first show that $F$ is 2-connected (which is required by Theorem \ref{theorem:2-connected main}).

    \begin{claim}\label{subgraph-F-is-2-connected-in-theorem}
        $F$ is 2-connected.
    \end{claim}
    
    \begin{proof}
        By Item 6 of Claim \ref{claim:simple properties of H',P}, for every $ij \in E(P)$ it holds that 
        $H_{ij}[V(H') \cap V(H_{ij})] = 
        A[V(H') \cap V(H_{ij})]$ is 2-connected. Since $f_{ij}$ is an isomorphism, the graph 
        $F_{ij} := H[f_{ij}^{-1}\left( 
        V(H') \cap V(H_{ij}) \right)]$ is also 2-connected.
        Note also that for every $ij \in E(P)$ it holds that $|V(H') \cap V(H_{ij}) \cap \{s',w_i,w_j\}| \geq 2$ (by Items 2,4 of Claim \ref{claim:simple properties of H',P}), and therefore $|V(F_{ij}) \cap L| \geq 2$ (as $f_{ij}^{-1}(\{s',w_i,w_j\}) = \{s,a,b\} \subseteq L$).

        By definition, $V(F)$ is the union of $L$ and $V(F_{ij})$ for $ij \in E(P)$. Hence, the union 
        $F^* := 
        L \cup \bigcup_{ij \in E(P)}F_{ij}$ is a spanning subgraph of $F$. It therefore suffices to show that $F^*$ is 2-connected. 
        
        We now consider two cases. Suppose first that $L$ is 2-connected. Then
        $F^* = 
        L \cup \bigcup_{ij \in E(P)}F_{ij}$ is the union of 2-connected graphs. Also, $|L \cap V(F_{ij})| \geq 2$ for all 
        $ij \in E(P)$, as we saw above. Hence, by repeated applications of Lemma \ref{claim:generalized-glue-on-edges}, we conclude that $F^*$ is 2-connected, as required.

        Now suppose that $L$ is not 2-connected. This is only possible if $\overline{L}$ is a triangle and $ab \notin E(H)$ (indeed, if $\overline{L}$ were 3-connected then $L$ would be 2-connected). By Items 2,5 of Claim \ref{claim:simple properties of H',P}, we have 
        $L = \{s,a,b\} \subseteq V(F_{ij})$ for all $ij \in E(P)$. Hence, we can omit $L$ from the definition of $F^*$, i.e., 
        $F^* = \bigcup_{ij \in E(P)}F_{ij}$. It also follows that $|V(F_{ij}) \cap V(F_{k\ell})| \geq 3$ for all $ij,k\ell \in E(P)$. Again, by repeated applications of Lemma \ref{claim:generalized-glue-on-edges}, we conclude that $F^*$ is 2-connected. This completes the proof. 
    \end{proof}
    
    Proceeding with the inductive step, we now split into the cases $F \neq H$ and $F = H$. In the former case we will apply the induction hypothesis, and in the latter we will obtain a contradiction directly.

\paragraph{Case 1:} 
$F \neq H$; namely,
$V(F) \subsetneq V(H)$. 
In this case, we would like to apply the induction hypothesis to the graph $F$ (in place of $H$). Thus, we need to check that all conditions of Theorem \ref{theorem:2-connected main} are satisfied. 
First, $F$ is 2-connected by Claim \ref{subgraph-F-is-2-connected-in-theorem}. Second, we need that $L$ is a leaf of $F$. However, there is one case where this can fail: if $V(F) = L$ and $\overline{L}$ is 3-connected but $L = \overline{L} - \{xy\}$ is not 3-connected, then $L$ need not be a leaf of $F$. Hence, we handle this case separately (see Case 1.2 below).

\paragraph{Case 1.1:}
$L \subsetneq V(F)$ or $\overline{L}$ is a triangle. 
We claim that in this case, $L$ is a leaf of $F$. Indeed, if $L \subsetneq V(F)$ then this holds, because $V(F) \setminus L \neq \emptyset$ and 
$\{x,y\}$ separates 
$L \setminus \{x,y\}$ from $V(F) \setminus L$. Also, if $L = V(F)$ and $\overline{L}$ is a triangle then $F$ is also a triangle (as $F$ is 2-connected), in which case $L$ is also a leaf of $F$. Hence, all conditions in Theorem \ref{theorem:2-connected main} are satisfied. 
 
We now observe two crucial points.
First, note that $A_F := A(G,F,s,a,b)$ can be identified with the subgraph of $A = A(G,H,s,a,b)$ induced by 
$$
\{s'\} \cup \{w_1,\dots,w_n\} \cup 
\bigcup_{ij \in E(G)} f_{ij}(V(F)).
$$
Second, by the definition of the graph $F$, it holds that 
$V(H') \subseteq V(A_F)$ and hence $H' \subseteq A_F$. 
Now, let $f : H \rightarrow H'$ be an isomorphism from $H$ to its copy $H'$, and let $F' := f(F)$ be the subgraph of $H'$ playing the role of $F$. Then $F'$ is a copy of $F$ in $A_F$. 
Moreover, since $H' \subseteq A' = A - \{s'w_i : i \in C\}$, we get that
$F' \subseteq A'_F := A_F - \{s'w_i : i \in C\}$. But this contradicts the induction hypothesis.

\paragraph{Case 1.2:}
$V(F) = L$ and $\overline{L}$ is 3-connected. In this case we proceed similarly as in the proof of Theorem \ref{theorem:main-theorem-H-equals-L}. 
As above, let $f : H \rightarrow H'$ be an isomorphism from $H$ to its copy $H'$, and let $L' := f(L)$ be the copy of $L$ in $H'$. 
Note that $P$ contains a minimal support $P_L$ for the graph $L'$ (as in Definition \ref{def:minimal edge-set}). Indeed, simply take $P_L$ to be a minimal subgraph of $P$ with the property that 
$L' \subseteq A_{P_L}$; this is well-defined because 
$L' \subseteq A_P$. 
First we claim that $e(P_L) \geq 2$. So suppose by contradiction that $e(P_L) \leq 1$. 
Note that $L' \cong L$ is 2-connected, because $\overline{L}$ is 3-connected by assumption. In particular, $\delta(L') \geq 2$, and hence $P_L$ has no isolated vertices by Item 2 of Fact \ref{fact:minimality of P}.
Hence, $P_L$ is an edge, say $P_L = \{ij\}$ for some $ij \in E(G)$.
Then $L' \subseteq H_{ij}$. Next, we use the assumption that $V(F) = L$. By the definition of $F$, this means that $L'$ can only contain vertices from $f_{ij}(L)$; i.e., $L' \subseteq f_{ij}(L)$. Note that $L'$ and $A[f_{ij}(L)]$ are both isomorphic to $L$, and so 
$L' = A[f_{ij}(L)]$. However, one of the edges $s'w_i,s'w_j$ (both of which belong to 
$A[f_{ij}(L)]$) was deleted when turning $A$ into $A'$. Hence, this contradicts that $L' \subseteq H'$ is a subgraph of $A'$. 

Thus, $e(P_L) \geq 2$. Note that $P_L$ is connected. Indeed, if not, then by Lemma \ref{claim:separators of H'} with $K := L'$ and $R := \emptyset$, we get that $L' - \{s'\}$ is disconnected, which is impossible because $L'$ is 2-connected (as observed above).
Now, since $P_L$ is a tree with at least 2 edges, there is $i \in V(P_L)$ with $d_{P_L}(i) \geq 2$. Then $R := \{i\}$ is a separator of both $P_L$ and $P$ (as $P_L \subseteq P$ and $P$ is also a tree). 
Moreover, we can enumerate the connected components of $P_L - \{i\}$ as $C^L_1,\dots,C^L_m$ and 
the connected components of $P - \{i\}$ as $C^H_1,\dots,C^H_{m'}$, where $m' \geq m \geq 2$, such that $C^L_j \subseteq C^H_j$ for all $j \in [m]$ (here again we use that $P_L \subseteq P$ are trees).
We now apply Lemma \ref{claim:separators of H'} to both $H'$ and $L'$. 
Put $R' := \{s',w_i\}$. 
Let $D^L_1,\dots,D^L_m$ be the sets defined in \eqref{eq:separators of H'} when applying Lemma \ref{claim:separators of H'} to $K := L'$. Similarly, let $D^H_1,\dots,D^H_{m'}$ be the sets defined in \eqref{eq:separators of H'} when applying Lemma \ref{claim:separators of H'} to $K := H'$.  
Then $D^L_j \subseteq D^H_j$ for all $j \in [m]$.
By Lemma \ref{claim:separators of H'} (applied to $L'$), $L'$ intersects each of the sets $D^L_1,\dots,D^L_m$ and hence also each of the sets
$D^H_1,\dots,D^H_{m}$.
Moreover, by Lemma \ref{claim:separators of H'} applied to $H'$,
there is no path in $H'-R'$ between any two of the sets $D^H_1,\dots,D^H_{m}$.
Hence, there are $u,v \in L' \setminus R'$ such that every $u,v$-path in $H'$ passes through $R'$. 
However, by Lemma \ref{claim:three-VDPs}, there are three internally vertex-disjoint $u,v$-paths in $H'$ 
(here we use the assumption that $\overline{L}$ is 3-connected). As $|R'| \leq 2$, this is a contradiction, completing the proof of Case 1.2.

\paragraph{Case 2:}
$F = H$. We will obtain a contradiction by showing that $H'$ contains strictly more separators of size 2 (i.e. 2-cuts) than $H$, and thus, $H'$ cannot be isomorphic to $H$. To this end, we construct an injective but non-surjective mapping from the set of 2-cuts of $H$ to the set of 2-cuts \nolinebreak of \nolinebreak $H'$.

Consider any 2-cut $\{u, v\}$ in $H$. We have 
$\{u,v\} \neq \{s,a\}$, as $\{s, a\}$ is not a 2-cut in $H$ by Lemma \ref{claim:sa-sb-not-2-cuts}. Let 
$q \in \{s,a\} \setminus \{u,v\}$; also, if $s \notin \{u,v\}$ then pick $q = s$.  
Let $p \in V(H) \setminus \{u,v\}$ be some vertex such that $p,q$ are in different connected components of $H - \{u, v\}$. 

\begin{claim}\label{claim:f_ij(z) in H'}
    There exists $ij \in E(P)$ such that $f_{ij}(p) \in V(H')$. 
\end{claim}
\begin{proof}
    First, suppose that $p \notin L$. Then 
    $p \in V(H) \setminus L = V(F) \setminus L$, so by definition of $V(F)$ there is some 
    $ij \in E(P)$ such that $p \in f_{ij}^{-1}(V(H') \cap V(H_{ij}))$; that is, $f_{ij}(p) \in V(H')$, as desired. Now, suppose that $p \in L$. If $\overline{L}$ is 3-connected then by Lemma \ref{claim:three-VDPs}, since $q \in \{s, a\} \subseteq L$, there are 3 internally vertex-disjoint $p,q$-paths in $H$. Hence, $p,q$ cannot be separated in $H$ by deleting two vertices, contradicting the choice of $p$. 
    And if $\overline{L}$ is a triangle, then by Items 2,5 of Claim \ref{claim:simple properties of H',P} we have $f_{ij}(p) \in f_{ij}(L) \subseteq V(H')$ for all 
    $ij \in E(P)$, as required.
\end{proof}

\begin{claim}\label{claim:f_ij(c) in H'}
    For every $ij \in E(P)$, it holds that 
    $f_{ij}(q) \in V(H')$.
\end{claim}
\begin{proof}
    Fix any $ij \in E(P)$.
    If $q = s$ then 
    $f_{ij}(s) = s' \in V(H')$ by 
    Item 2 of Claim \ref{claim:simple properties of H',P}. So suppose that $q \neq s$. By our choice of $q$, this means that $q = a$ and $s \in \{u,v\}$, say $u = s$. 
    Note that the case where $\overline{L}$ is 3-connected is impossible because by Lemma \ref{lemma:sv-not-2-cut-if-H-is-3-connected}, if $\overline{L}$ is 3-connected then
    $\{s,v\}$ is not a 2-cut of $H$ for any $v \in V(H)$.
    And if $\overline{L}$ is a triangle, then by Item 5 of Claim \ref{claim:simple properties of H',P}, it holds that 
    $f_{ij}(q) = f_{ij}(a) \in \{w_i,w_j\} \subseteq V(H')$, as required. 
\end{proof}

By Claims \ref{claim:f_ij(z) in H'} and \ref{claim:f_ij(c) in H'}, there exists 
$ij \in E(P)$ such that $f_{ij}(p), f_{ij}(q) \in V(H')$. Recall that $S := \{u,v\}$ separates $p,q$ in $H$.
Hence, by Lemma \ref{lemma:H'-contains-intermediate-vertices-generalized} (with $k=2$), we have that 
$f_{ij}(u), f_{ij}(v) \in V(H')$ and
$\{f_{ij}(u), f_{ij}(v)\}$ is a separator of $H'$.
This proves that for every 2-cut $\{u, v\}$ of $H$, there exists some $ij \in E(P)$ such that 
$\{f_{ij}(u), f_{ij}(v)\}$ is a 2-cut of $H'$.
This defines a mapping from the set of 2-cuts of $H$ to the set of 2-cuts of $H'$.
Furthermore, this map is injective. Indeed, if two distinct 2-cuts
$\{u,v\},\{u',v'\}$ of $H$ are mapped to the same pair, then there must be distinct edges 
$ij,k\ell \in E(G)$ such that $\{f_{ij}(u),f_{ij}(v)\} = \{f_{k\ell}(u'),f_{k\ell}(v')\}$ (indeed, if $ij = k\ell$ then $\{u,v\},\{u',v'\}$ are mapped to different pairs because $f_{ij}$ is a bijection). However, $V(H_{ij}) \cap V(H_{k\ell}) = \{s'\} \cup (\{w_i,w_j\} \cap \{w_k,w_{\ell}\})$. This means that $\{f_{ij}(u),f_{ij}(v)\} = \{f_{k\ell}(u'),f_{k\ell}(v')\} = \{s',w_t\}$ for some $t \in V(G)$, and hence $\{u,v\},\{u',v'\} \in \{\{s,a\}, \{s,b\}\}$. 
But $\{s,a\},\{s,b\}$ are not 2-cuts of $H$ by Lemma \ref{claim:sa-sb-not-2-cuts}. This proves injectivity. 

Now, we claim that there exists an additional 2-cut in $H'$ which is not obtained in the above way.
To see this, let $i \in V(P)$ be such that $d_P(i) \geq 2$; such an $i$ exists because $P$ is a tree with at least 2 edges. By Item 3 of Claim \ref{claim:simple properties of H',P}, $\{s',w_i\}$ is a 2-cut of $H'$. But observe that this 2-cut is not the image of any 2-cut of $H$, i.e., there is no 2-cut $\{u,v\}$ of $H$ such that $\{s',w_i\} = \{f_{ij}(u),f_{ij}(v)\}$ for some $ij \in E(P)$. Indeed, for every $ij \in E(P)$,
$\{ f^{-1}_{ij}(s'), f^{-1}_{ij}(w_i) \}$ equals 
$\{s,a\}$ or $\{s,b\}$, but these are not 2-cuts of $H$, by Lemma \ref{claim:sa-sb-not-2-cuts}.
Finally, we see that $H'$ has more 2-cuts than $H$, giving the desired contradiction. 
This completes the proof of Case 2, and hence the proof of the theorem. 
\end{proof}

\section{Connected $H$ with $\delta(H) \geq 2$}\label{sec:connected, delta >=2}
The goal of this section is to extend Theorem \ref{theorem:2-connected H} to all connected graphs $H$ with minimum degree at least 2. As we shall see, Construction \ref{construction:general-construction} is not suitable to handle all such graphs, and so in some cases we will need a variant of this construction (see Construction \ref{construction:second construction} below). 

Our approach for treating connected graphs $H$ with $\delta(H) \geq 2$ is to consider the decomposition of $H$ into biconnected components. Recall that this decomposition can be described using a tree, whose nodes correspond to the biconnected components and to cut vertices, see \cite[Section 3.1]{Diestel2017}. Alternatively, this is a tree-decomposition with adhesion 1, where each bag induces a 2-connected graph, an edge, or a single vertex. The following definition is analogous to Definition \ref{def:leaf}, only for adhesion 1. 

\begin{definition}[leaf biconnected component]
    Let $H$ be a connected graph. 
    A {\em leaf biconnected component} of $H$ is a subset $M \subseteq V(H)$ such that $H[M]$ is 2-connected and either 
    $M = V(H)$, or there is a vertex 
    $z \in M$ which separates $M$ from $V(H) \setminus M$.\footnote{With a slight abuse of notation, we will denote by $M$ both the vertex-set and the induced subgraph $H[M]$.} 
\end{definition}



\begin{fact}\label{fact:leaf biconnected component}
    Every connected graph $H$ with $\delta(H) \geq 2$ has a leaf biconnected component. 
\end{fact}
\begin{proof}
    Consider the decomposition of $H$ into biconnected components, and take $M$ to be the vertex-set of a biconnected component corresponding to a leaf in the corresponding tree (unless $H$ is 2-connected, in which case take $M = V(H)$). If $M$ is not 2-connected, meaning that $M$ is an edge or a single vertex, then $\delta(H) \leq 1$, a contradiction.
\end{proof}


\begin{definition}[good choice]
For a 2-connected graph $M$, we say that a triple $s,a,b$ is a {\em good choice} for $M$ if it satisfies the conclusion of Theorem \ref{theorem:2-connected H}.
\end{definition}

\noindent
When applying Construction \ref{construction:general-construction} to a connected graph $H$ with $\delta(H) \geq 2$, we will proceed as follows: Pick (arbitrarily) a leaf biconnected component $M$ of $H$.
In the case that $M \neq H$, we will denote by $z$ the cut vertex which separates $M$ from $V(H) \setminus M$. The following theorem describes the conditions under which we can apply Construction \ref{construction:general-construction} in this setting. The crucial condition is that $z \neq a,b$.

\begin{theorem}\label{theorem:connected H, delta >= 2, type 1}
    Let $H$ be a connected graph with $\delta(H) \geq 2$. Let $M$ be a leaf biconnected component of $H$, and, if $M \neq H$, let $z$ be the cut vertex separating $M$ from $V(H) \setminus M$. 
    Let $s,a,b \in M$ be a good choice for $M$ with $s \neq z$ (if $z$ is defined); such a choice exists by Theorem \ref{theorem:2-connected H}. 
    Suppose that $z \neq a,b$ (if $z$ is defined). 
    Let $G$ be a forest with $V(G) = [n]$, and let 
    $A = A(G, H, s, a, b)$ as in Construction \ref{construction:general-construction}. Let $C$ be a vertex-cover of $G$. Then 
    $A' := A - \{s'w_i : i \in C\}$ is $H$-free.
\end{theorem}

Before proving Theorem \ref{theorem:connected H, delta >= 2, type 1}, we describe the construction which we use for the remaining graphs $H$. The crucial assumption in Theorem \ref{theorem:connected H, delta >= 2, type 1} is that $z \neq a,b$. There are, however, graphs where such a choice of $s,a,b$ is impossible, and furthermore, for which the conclusion of Theorem \ref{theorem:connected H, delta >= 2, type 1} does \nolinebreak not \nolinebreak hold.

\begin{example}
    Let $H$ be the graph consisting of two triangles sharing a vertex; namely, $V(H) = \{v_1,v_2,v_3,v_4,v_5\}$ and 
    $T_1 := \{v_1,v_2,v_3\}, 
    T_2 := \{v_3,v_4,v_5\}$ are triangles (so $T_1,T_2$ are the biconnected components of $H$). 
    Suppose we used Construction \ref{construction:general-construction} for this graph $H$, with $s,a,b$ being the vertices of one of the two triangles; say 
    $\{s,a,b\} = V(T_1)$. For every edge $ij \in E(G)$, the copy $H_{ij}$ in the construction includes a copy of the triangle $T_2$, and this triangle is ``hanging" on one of the vertices playing the roles of $s,a,b$, i.e., on $s'$, $w_i$ or $w_j$. Since vertices 
    $i \in V(G)$ may be incident to many edges, we obtain many such hanging triangles on some vertex in the set $\{s'\} \cup \{w_1,\dots,w_n\}$. In particular, any two of these triangles (hanging on the same vertex) form a copy of $H$. However, these copies contain no edges of the form $s'w_i$, $i \in [n]$, so it is no longer true that one can make the construction $H$-free by deleting only edges of this type. 
\end{example}

The solution to the above issue is to ``hang" only one copy of $H - (M \setminus \{z\})$ on each vertex $w_i$. (In the above example, 
$H - (M \setminus \{z\})$ is the triangle $T_2$.)
This leads to the following construction. 
\begin{construction}[$B(G,H,M,s,a,b)$]\label{construction:second construction}
Let $H$ be a connected graph with $\delta(H) \geq 2$. Let $M$ be a leaf biconnected component of $H$, and, if $M \neq H$, let $z$ be the cut vertex separating $M$ from $V(H) \setminus M$. Let $s,a,b \in M$ be distinct vertices with $sa,sb \in E(H)$. Let $G$ be a graph with $V(G) = [n]$. Define a graph $B = B(G,H,M,s,a,b)$ as follows:
\begin{itemize}
    \item Start with the construction 
    $A = A(G,M,s,a,b)$. We will denote by $M_{ij}$ the copy of $M$ which corresponds to an edge $ij \in E(G)$.\footnote{In Construction \ref{construction:general-construction}, these were denoted by $H_{ij}$.} 
    \item Assuming that $z$ is defined, for each $i \in [n]$, add a copy $H_i$ of 
    $N := H - (M \setminus \{z\})$ in which $w_i$ plays the role of $z$ (where $w_i$ is the vertex of $A$, as in Construction \ref{construction:general-construction}) and all other vertices are new (i.e., are disjoint from the vertices of $A$ and from all other copies). Denote by 
    $g_i : N \rightarrow H_i$ the corresponding isomorphism, so that $g_i(z) = w_i$. 
\end{itemize}
\end{construction}
\noindent
We will refer to the copies $H_i$ also in the case where $M = H$; in this case $H_i = \{w_i\}$. See Figure \ref{fig:B construction} for an illustration of Construction \ref{construction:second construction}.

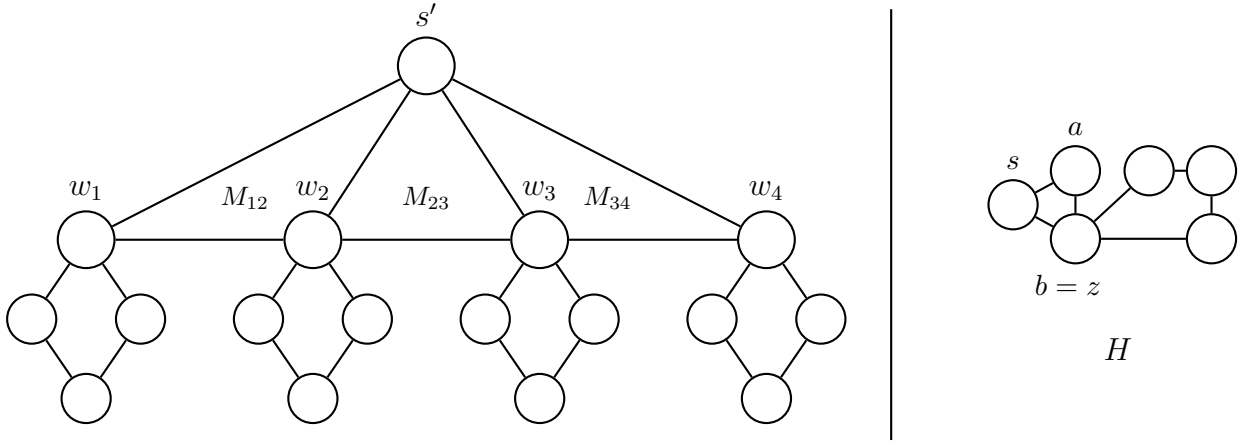
\begin{figure}[h]
    \centering
    \begin{tikzpicture}[
  vertex/.style={circle,draw,thick,minimum size=7.5mm,inner sep=0pt,fill=white},
  smallvertex/.style={circle,draw,thick,minimum size=6.5mm,inner sep=0pt,fill=white},
  edge/.style={draw,thick},
  labelstyle/.style={font=\large},
  sidetitle/.style={font=\large\bfseries}
]

\node[vertex,label={[labelstyle]above:$s'$}] (sp) at (0,4.5) {};

\node[vertex,label={[labelstyle]above:$w_1$}] (w1) at (-4.5,2.2) {};
\node[vertex,label={[labelstyle]above:$w_2$}] (w2) at (-1.5,2.2) {};
\node[vertex,label={[labelstyle]above:$w_3$}] (w3) at ( 1.5,2.2) {};
\node[vertex,label={[labelstyle]above:$w_4$}] (w4) at ( 4.5,2.2) {};

\foreach \w in {w1,w2,w3,w4}{\draw[edge] (sp)--(\w);}
\draw[edge] (w1)--(w2)--(w3)--(w4);

\foreach \i/\x in {1/-4.5,2/-1.5,3/1.5,4/4.5}{
  \node[smallvertex] (l\i) at ($(w\i)+(-0.72,-1.05)$) {};
  \node[smallvertex] (d\i) at ($(w\i)+(0,-2.10)$) {};
  \node[smallvertex] (r\i) at ($(w\i)+(0.72,-1.05)$) {};
  \draw[edge] (w\i)--(l\i)--(d\i)--(r\i)--(w\i);
}

\node[font=\small] at (-2.4,2.75) {$M_{12}$};
\node[font=\small] at ( 0.0,2.75) {$M_{23}$};
\node[font=\small] at (2.4,2.75) {$M_{34}$};

\draw[thick] (6.15,-0.45)--(6.15,5.25);

\begin{scope}[scale = 0.75, xshift=3cm]
\node[sidetitle] at (9.20,1) {$H$};

\node[smallvertex,label=above:$s$] (Hs) at (7.35,3.55) {};
\node[smallvertex,label=above:$a$] (Ha) at (8.45,4.15) {};
\node[smallvertex,label={[xshift=-1mm]below:$b=z$}] (Hz) at (8.45,2.95) {};
\draw[edge] (Hs)--(Ha)--(Hz)--(Hs);

\node[smallvertex] (Hu) at (9.75,4.15) {};
\node[smallvertex] (Hv) at (10.85,4.15) {};
\node[smallvertex] (Ht) at (10.85,2.95) {};
\draw[edge] (Hz)--(Hu)--(Hv)--(Ht)--(Hz);
\end{scope}
\end{tikzpicture}
    \caption{The construction $B = B(G,H,M,s,a,b)$ with $G = \{12,23,34\}$ and the graph $H$ depicted on the right.}
    \label{fig:B construction}
\end{figure}

\begin{fact}\label{fact:copies-of-H-second-construction}
    Let $ij \in E(G)$, $i < j$. Let $M,z$ be as above, and suppose that if $M \neq H$ (so that $z$ is defined) then
    $z \in \{a, b\}$. Then, $B[M_{ij} \cup H_i] \cong H$ or $B[M_{ij} \cup H_j] \cong H$. 
\end{fact}
\begin{proof}
    If $M = H$ then clearly $M_{ij}$ is a copy of $H$ (and 
    $H_i = \{w_i\},H_j = \{w_j\} 
    \subseteq M_{ij}$). 
    Assume now that $M \neq H$.
    Suppose $z = a$. 
    Then $f_{ij}(z) = f_{ij}(a) = w_i$. 
    Now, $M_{ij}$ is a copy of $M$ where $w_i$ plays the role of $a = z$, and $H_i$ is a copy of $H - (M \setminus \{z\})$ where $w_i$ plays the role of $z$. Hence, $M_{ij} \cup H_i$ forms a copy of $H$. The case of $z = b$ is identical (taking $M_{ij} \cup H_j$ in place of $M_{ij} \cup H_i$). 
\end{proof}

\noindent
We now prove versions of Facts 
\ref{lemma:edges-belong-uniquely} and \ref{fact:subgraph} and
Lemmas \ref{lemma:easy-inequality} and \ref{thm:g-locally-a-forest} for Construction \ref{construction:second construction}. 

\begin{fact}\label{lemma:edges-belong-uniquely, type 2}
    Let $B = B(G,H,M,s,a,b)$, and let $e := vw \in E(B) \setminus \{s'w_i : i \in [n]\}$. Then, either there exists a unique edge $ij \in E(G)$ such that $e \in E(M_{ij})$, or there exists a unique $i \in V(G)$ such that $e \in E(H_i)$.
\end{fact}

\begin{proof}
    By definition, $B$ consists of the construction $A = A(G,M,s,a,b)$ and the ``hanging" copies $H_i$ of 
    $H - (M \setminus \{z\})$ (in the case $M \neq H$). 
    If $e \in E(A)$ then by Fact \ref{lemma:edges-belong-uniquely}, there exists a unique edge $ij \in E(G)$ such that $e \in E(M_{ij})$. 
    Otherwise, $e \in E(H_i)$ for some $i \in [n] = V(G)$. 
    Clearly, the graphs $H_1,\dots,H_n$ are vertex-disjoint, so $i$ is unique as desired. 
\end{proof}

\begin{fact}\label{fact:subgraph, type 2}
    Let $B = B(G,H,M,s,a,b)$, and let $P$ be a subgraph of $G$. Then $B_P := B(P,H,M,s,a,b)$ can be identified with the subgraph of $B$ having vertex-set
    \begin{equation}\label{eq:subgraph of B}
    V(B_P) := \{s'\} \cup 
    \Big( \bigcup_{i \in V(P)}V(H_i) \Big) \cup
    \Big( \bigcup_{ij \in E(P)}V(M_{ij}) \Big)
    \end{equation}
    and edge-set
    \begin{equation}\label{eq:subgraph of B, 2}
    \{s'w_i : i \in V(P)\} \cup 
    \Big( \bigcup_{i \in V(P)} E(H_i) \Big) 
    \cup 
    \Big( \bigcup_{ij \in E(P)}E(M_{ij}) \Big).
    \end{equation}
\end{fact}
Again, with a slight abuse of notation, we will sometimes write 
$B_P = B(P,H,M,s,a,b)$ for the subgraph of 
$B = B(G,H,M,s,a,b)$ with vertex-set and edge-set defined in \eqref{eq:subgraph of B} and \eqref{eq:subgraph of B, 2}, respectively. (By Fact \ref{fact:subgraph, type 2}, $B_P$ is indeed isomorphic to $B(P,H,M,s,a,b)$.)

\begin{lemma}\label{lemma:easy-inequality, type 2}
    Let $G$ be a graph on $[n]$, let $H,M,z,s,a,b$ be as in Construction \ref{construction:second construction}, and let 
    $B = B(G,H,M,s,a,b)$. 
    Suppose that $M = H$ or $z \in \{a,b\}$.
    Then 
    $\Rem_H(B) \geq \tau(G)$.
\end{lemma}


\begin{proof}
    We proceed similarly to the proof of Lemma \ref{lemma:easy-inequality}. Let $S \subseteq E(B)$ such that $B - S$ is $H$-free. We define another edge-set $S' \subseteq E(B)$ as follows. For each edge $e \in S$, if $e = s'w_i$ for some $i \in V(G)$, add $e$ to $S'$. Otherwise, by Fact \ref{lemma:edges-belong-uniquely, type 2}, there exists either a unique $ij \in E(G)$ with $e \in E(M_{ij})$ or a unique $i \in V(G)$ with $e \in E(H_i)$. In the former case, add $s'w_i$ or $s'w_j$ to $S'$ arbitrarily, and in the latter case, add $s'w_i$ to $S'$. This completes the definition of $S'$. Note that $|S'| \leq |S|$. Moreover, if for some $ij \in E(G)$ it holds that $s'w_i, s'w_j \not \in S'$, then $(E(M_{ij}) \cup E(H_i) \cup E(H_j)) \cap S = \emptyset$, which is impossible as $B - S$ is $H$-free, and either $B[M_{ij} \cup H_i] \cong H$ or $B[M_{ij} \cup H_j] \cong H$ by Fact \ref{fact:copies-of-H-second-construction}. Thus, $C := \{i \in V(G) : s'w_i \in S'\}$ is a vertex-cover of $G$, and so $|S| \geq |S'| \geq \tau(G)$ as desired.
\end{proof}


\begin{lemma}\label{thm:g-locally-a-forest, type 2}
    Let $B = B(G, H, M, s, a, b)$ and let $H'$ be a copy of $H$ in $B$. 
    Then there exists a subgraph $F \subseteq G$ with $v(F) \leq 2v(H)$ such that $H'$ is a subgraph of 
    $B_F = B(F,H,M,s,a,b)$.
\end{lemma}

\begin{proof}
    The proof is similar to that of Lemma \ref{thm:g-locally-a-forest}. 
    Namely, define the subgraph $F \subseteq G$ as follows: For each $i \in V(G)$ such that 
    $V(H') \cap V(H_i) \neq \emptyset$, put $i$ into $V(F)$. 
    In particular, if $w_i \in V(H')$ then $i \in V(F)$.
    For each $ij \in E(G)$ such that 
    $w_iw_j \in E(H')$, put the edge $ij$ into $E(F)$ (the vertices $i,j$ are already in $V(F)$ by the above).
    Finally, for each $ij \in E(G)$ such that 
    $$
    V(H') \cap (V(M_{ij}) \setminus \{s',w_i,w_j\}) \neq \emptyset,
    $$
    put both vertices $i,j$ into $V(F)$ and the edge $ij$ into $E(F)$. It is easy to see that 
    $v(F) \leq 2v(H') = 2v(H)$,
    $V(H') \subseteq V(B_F)$, and $E(H') \subseteq E(B_F)$.
\end{proof}

\noindent
The second main result of this section is as follows:
\begin{theorem}\label{theorem:connected H, delta >= 2, type 2}
    Let $H$ be a connected graph with $\delta(H) \geq 2$. Let $M$ be a leaf biconnected component of $H$, and, if $M \neq H$, let $z$ be the cut vertex separating $M$ from $V(H) \setminus \nolinebreak M$. Let $s,a,b \in M$ be a good choice for $M$ with $s \neq z$ (if $z$ is defined); such a choice exists by Theorem \ref{theorem:2-connected H}. 
    Suppose that $z \in \{a,b\}$ (if $z$ is defined). 
    Let $G$ be a forest with $V(G) = [n]$, and let 
    $B = B(G, H, M, s, a, b)$ as in Construction \ref{construction:second construction}. Let $C$ be a vertex-cover of $G$. Then 
    $B' := B - \{s'w_i : \nolinebreak i \in \nolinebreak C\}$ is $H$-free.
\end{theorem}

We see that the conditions in Theorems \ref{theorem:connected H, delta >= 2, type 1} and \ref{theorem:connected H, delta >= 2, type 2} are complementary (i.e., 
$z \notin \{a,b\}$ and $z \in \{a,b\}$).
The proofs of these theorems are given in Sections 
\ref{subsec:type 1} and \ref{subsec:type 2}, respectively. 

\subsection{Proof of Theorem \ref{theorem:connected H, delta >= 2, type 1}}\label{subsec:type 1}
\begin{proof}[Proof of Theorem \ref{theorem:connected H, delta >= 2, type 1}]
The proof is similar to that of Theorem \ref{theorem:2-connected main} (and in fact somewhat simpler). We proceed by induction on $|V(H)|$. The base case is that $H$ is 2-connected (and hence $M = H$); this case is immediate, because the conditions of Theorem \ref{theorem:connected H, delta >= 2, type 1} require that the triple $s,a,b$ is a good choice for $M = H$, which means that the conclusion of Theorem \ref{theorem:2-connected H} holds.

Suppose now that $H$ is not 2-connected, and let $M,z$ be as in the statement of the theorem; in particular, $z \neq a,b$.
Let $A = A(G,H,s,a,b)$ and
$A' = A - \{s'w_i : i \in C\}$, and suppose by contradiction that $A'$ contains a copy $H'$ of $H$. 
Let $F$ be the subgraph of $H$ induced on the vertex \nolinebreak set
$$
V(F) := M \cup \bigcup_{ij \in E(G)}
f_{ij}^{-1}(V(H') \cap V(H_{ij})).
$$
If $F \neq H$ then we will apply the induction hypothesis for the graph $F$, and if $F = H$ then we will obtain a contradiction directly. First we need to verify that $F$ satisfies the conditions of Theorem \ref{theorem:connected H, delta >= 2, type 1}, i.e., that $F$ is connected and $\delta(F) \geq 2$. 
    

\begin{claim}\label{claim:type 1, F properties}
$F$ is connected and $\delta(F) \geq 2$.
\end{claim}
\begin{proof}
    We first show that $F$ is connected. 
    We could do this by using Lemmas \ref{lemma:intersections with H_{ij} are k-connected} and \ref{claim:generalized-glue-on-edges} (as in the proof of Theorem \ref{theorem:2-connected main}), but here the situation is simple enough to argue directly, as follows.
    By definition, we have $M \subseteq V(F)$, 
    and $M$ is connected. 
    Hence, it suffices to show that every 
    $v \in V(F) \setminus M$ has a path in $F$ to some vertex of $M$. 
    By the definition of $F$, there must exist some 
    $ij \in E(G)$ such that $f_{ij}(v) \in V(H')$ (as $v \in V(F) \setminus M$). By Lemma \ref{lem:|P|>=2}, $V(H') \not \subseteq V(H_{ij})$ (here we use the assumption that $\delta(H) \geq 2$). Since $H'$ is connected (as a copy of $H$), there must exist a path in $H'$ from $f_{ij}(v)$ to $V(H') \setminus V(H_{ij})$. But by Fact \ref{fact:outside vertices}, the only vertices in $V(H_{ij})$ which may have neighbors outside of $V(H_{ij})$ are $s',w_i,w_j$. Hence, there is a path $Q$ inside $H'[V(H_{ij})]$ from $f_{ij}(v)$ to $f_{ij}(M)$. Now $f_{ij}^{-1}(Q)$ is a path from $v$ to $M$, and 
    $f_{ij}^{-1}(Q) \subseteq 
    f_{ij}^{-1}(V(H') \cap V(H_{ij})) \subseteq V(F)$, \nolinebreak as \nolinebreak required. 
    

    Next we prove that $\delta(F) \geq 2$. So let $v \in V(F)$ and let us show that $d_F(v) \geq 2$. If $v \in M$ then this holds because $d_F(v) \geq d_M(v) \geq 2$ (as $M \subseteq V(F)$ and $M$ is 2-connected). Now suppose that $v \in V(H) \setminus M$. Since $v \in V(F)$, there exists some $ij \in E(G)$ such that $f_{ij}(v) \in V(H')$. We have $\delta(H') \geq 2$ (because $H'$ is a copy of $H$). Also, 
    $f_{ij}(v) \notin \{s',w_i,w_j\}$ (because $v \notin M$), and so (by Fact \ref{fact:outside vertices}), the only neighbors of $f_{ij}(v)$ in $A$ are inside $V(H_{ij})$. It follows that 
    $|N(f_{ij}(v)) \cap V(H') \cap V(H_{ij})| \geq 2$. 
    Applying $f_{ij}^{-1}$ to the set 
    $N(f_{ij}(v)) \cap V(H') \cap V(H_{ij})$, we get a set of at least two neighbors of $v$ in $F$, as required. 
\end{proof}

\noindent
We are now ready to proceed with the heart of the proof.

\paragraph{Case 1:} $F \neq H$; namely, $V(F) \subsetneq V(H)$. In this case we apply the induction hypothesis. Note that all conditions of Theorem \ref{theorem:connected H, delta >= 2, type 1} are satisfied. Indeed, $F$ is connected and $\delta(F) \geq 2$ (by Claim \ref{claim:type 1, F properties}); $M$ is a leaf biconnected component of $F$, and if $F \neq M$ then $z$ separates $M$ from $V(F) \setminus M$; and $s,a,b \in M$ is a good choice for $M$ with $z \notin \{s,a,b\}$. 

Observe that $A_F := A(G,F,s,a,b)$ can be identified with the subgraph of $A = A(G,H,s,a,b)$ induced by 
$$
\{s'\} \cup \{w_1,\dots,w_n\} \cup 
\bigcup_{ij \in E(G)} f_{ij}(V(F)).
$$
Moreover, by the definition of the graph $F$, it holds that 
$V(H') \subseteq V(A_F)$. 
Now, let $f : H \rightarrow H'$ be an isomorphism from $H$ to its copy $H'$, and let $F' := f(F)$ be the copy of $F$ in $H'$. Then $F'$ is a copy of $F$ in $A'_F := A_F - \{s'w_i : i \in C\}$, contradicting the induction hypothesis. This completes the proof in Case 1.

\paragraph{Case 2:} $F = H$. 
By the definition of $F$, this means that for each $v \in V(H) \setminus M$, there exists some $ij \in E(G)$ such that $f_{ij}(v) \in V(H')$. 
    
    \begin{claim}\label{claim:unique-i-1-connected-case-1}
        There exists at most one $ij \in E(G)$ such that $f_{ij}(V(H) \setminus M) \cap V(H') \neq \emptyset$.
    \end{claim}

    \begin{proof}

        Let $I$ be the set of all $ij \in E(G)$ such that 
        $f_{ij}(V(H) \setminus M) \cap V(H') \neq \emptyset$. 
        To prove that $|I| \leq 1$, we will show that 
        if $|I| \geq 2$ then $H'$ contains more 1-cuts (i.e., cut vertices) than $H$, which is impossible as $H'$ is a copy of $H$. Let $W$ be the set of all $w \in V(H)$ which are cut vertices of $H$. In particular, $z \in W$. 

        We will prove the following two claims:
        \begin{enumerate}
            \item[(i)] For every $w \in W \setminus \{z\}$, there exists $ij \in E(G)$ such that $f_{ij}(w)$ is a cut vertex of $H'$. 
            \item[(ii)] For every $ij \in I$, $f_{ij}(z)$ is a cut vertex of $H'$. 
        \end{enumerate}
        Note that all cut vertices of $H'$ obtained by Items (i)-(ii) are distinct. Indeed, if 
        $f_{ij}(w) = f_{k\ell}(w')$ for some $w,w' \in W$ and edges 
        $ij,k\ell \in E(G)$, then 
        $ij = k\ell$ because $f_{ij}(w),f_{k\ell}(w')$ are internal vertices of the basic copies $H_{ij},H_{k\ell}$ (here we use the assumption that $z \notin \{s,a,b\}$) and then $w=w'$ (because $f_{ij}$ is injective). 
        Hence, combining (i) and (ii), we conclude that $H'$ has at least $(|W|-1) + |I|$ cut vertices. 
        On the other hand, $H'$ has exactly $|W|$ cut vertices (since $H' \cong H$). Thus, $|I| \leq 1$, \nolinebreak as \nolinebreak required. 

        It remains to prove (i) and (ii). We begin with (i). Let $w \in W \setminus \{z\}$, and let $p \in V(H)$ be such that $M$ and $p$ are in different components of $H - w$. Note that 
        $w \notin M$ (since $z$ is the only cut vertex of $H$ contained in $M$). Since $p \notin M$, there exists $ij \in E(G)$ such that 
        $f_{ij}(p) \in V(H')$ (as explained above). 
        By Lemma \ref{lem:|P|>=2}, $V(H') \not\subseteq V(H_{ij})$ (here we use the assumption that $\delta(H) \geq 2$). Also, by Fact \ref{fact:outside vertices}, $s',w_i,w_j$ are the only vertices in $V(H_{ij})$ which may have neighbors outside of $V(H_{ij})$. Since $H'$ is connected, it follows that 
        $V(H') \cap \{s',w_i,w_j\} \neq \emptyset$. In particular, there exists $q \in \{s,a,b\} \subseteq M$ such that $f_{ij}(q) \in V(H')$. 
        Note that $w$ separates $p,q$ in $H$.
        Hence, by Lemma \ref{lemma:H'-contains-intermediate-vertices-generalized} (with parameter $k=1$), we have that $f_{ij}(w)$ is a separator of $H'$, as required. This proves (i).

        For (ii), fix any $ij \in I$. By the definition of $I$, there exists 
        $p \in V(H) \setminus M$ such that $f_{ij}(p) \in V(H')$. In the previous paragraph we already proved that there also exists 
        $q \in \{s,a,b\}$ such that $f_{ij}(q) \in V(H')$. Since $z \neq s,a,b$, we see that $z$ separates $p$ and $q$ in $H$ (as $p \in V(H) \setminus M$ and 
        $q \in M \setminus \{z\}$). Now, by Lemma \ref{lemma:H'-contains-intermediate-vertices-generalized}, $f_{ij}(z)$ is a separator of $H'$, proving (ii). 
    \end{proof}
    We now use Claim \ref{claim:unique-i-1-connected-case-1} to obtain a contradiction. 
    Let $\mathcal{M}(H)$ denote the set of copies of $M$ in $H$, and let $\mathcal{M}(H')$ denote the set of copies of $M$ in $H'$. 
    \begin{claim}\label{claim:type 1, M-copies}
        $|\mathcal{M}(H)| > |\mathcal{M}(H')|$.
    \end{claim}
    \begin{proof}
        Observe that $A_M := A(G,M,s,a,b)$ can be identified with the subgraph of
        $A = A(G,H,s,a,b)$ induced by
        $$
        \{s'\} \cup \{w_1,\dots,w_n\} \cup \bigcup_{ij \in E(G)} f_{ij}(M).
        $$
        Similarly, 
        $A'_M := A_M - \{s'w_i : i \in C\}$ is the subgraph of $A'$ induced by the same vertex set (and via the same identification). 
        Since $s,a,b$ is a good choice for $M$, the graph $A'_M$ is $M$-free. Hence, $H'$ has no copies of $M$ which are contained in $V(A'_M) = V(A_M)$. Also, by Claim
        \ref{claim:unique-i-1-connected-case-1}, there is at most one 
        $ij \in E(G)$ such that $H'$ contains vertices from $f_{ij}(V(H) \setminus M)$. 
        This means that $V(H') \subseteq V(A_M) \cup f_{ij}(V(H) \setminus M)$.
        Also, $V(H')$ must contain vertices from $V(A_M) \setminus \{f_{ij}(z)\}$, because 
        $f_{ij}(V(H) \setminus (M \setminus \{z\}))$ is too small to contain $V(H')$, as 
        $|V(H) \setminus 
        (M \setminus \{z\})| < |V(H)|$.
        Now, the vertex $f_{ij}(z)$ separates the sets $V(A_M)$ and $f_{ij}(V(H) \setminus M)$ in $A$ (and hence also in $H'$, in the case that $H'$ contains vertices from $f_{ij}(V(H) \setminus M)$). Since $M$ is 2-connected and $A'_M$ is $M$-free, every copy of $M$ in $H'$ is contained in 
        $f_{ij}(V(H) \setminus (M \setminus \{z\}))$. But 
        $V(H) \setminus (M \setminus \{z\})$ clearly contains at most $|\mathcal{M}(H)| - 1$ copies of $M$ (because $H$ contains $M$ itself). This proves the \nolinebreak claim. 
    \end{proof}
    \noindent
    Claim \ref{claim:type 1, M-copies} contradicts the fact that $H'$ is isomorphic to $H$, completing the proof of the theorem. 
\end{proof}

\subsection{Proof of Theorem \ref{theorem:connected H, delta >= 2, type 2}}\label{subsec:type 2}
\begin{proof}[Proof of Theorem \ref{theorem:connected H, delta >= 2, type 2}]
The proof is similar to that of Theorem \ref{theorem:connected H, delta >= 2, type 1}, so we will be brief in some places. The proof is by induction on $|V(H)|$, and the base case is that $H$ is 2-connected (and hence $M = H$); this case is immediate, because the triple $s,a,b$ is assumed to be a good choice for $M = H$, and in this case $B(G,H,M,s,a,b) = A(G,H,s,a,b)$.

Suppose now that $H$ is not 2-connected, and let $M,z$ be as in the statement of the theorem.
Let $B = B(G,H,M,s,a,b)$, and recall that 
$B$ consists of $A := A(G,M,s,a,b)$ and, for each $i \in [n]$, of a copy $H_i$ of 
$N := H - (M \setminus \{z\})$ ``hanging" on the vertex $w_i \in A$, which plays the role of $z$ in this copy. Also, 
$g_i : N \rightarrow H_i$ denotes the corresponding isomorphism (so $g_i(z) = w_i$).
Let
$B' = B - \{s'w_i : \nolinebreak i \in \nolinebreak C\}$, and suppose by contradiction that $B'$ contains a copy $H'$ of $H$. 
Let $F$ be the subgraph of $H$ induced by the vertex set
$$
V(F) := M \cup \bigcup_{i \in [n]}
g_i^{-1}(V(H') \cap V(H_i)).
$$
\begin{claim}\label{claim:type 2, F properties}
$F$ is connected and $\delta(F) \geq 2$.
\end{claim}
\begin{proof}
    We first show that $F$ is connected. By definition, we have $M \subseteq F$, and $M$ is connected. Hence, it suffices to show that every 
    $v \in V(F) \setminus M$ has a path in $F$ to some vertex of $M$. 
    By the definition of $F$, there must exist some 
    $i \in [n]$ such that $g_i(v) \in V(H')$ (because $v \in V(F) \setminus M$). 
    Note that $V(H') \not \subseteq V(H_i)$, since $|V(H_i)| < |V(H)| = |V(H')|$. Also, $w_i = g_i(z)$ is the only vertex of $H_i$ with neighbors outside of $H_i$. Since $H'$ is connected, it follows that there is a path $Q$ in $H'[V(H_i)]$ from $g_i(v)$ to $g_i(z)$. Then $g_i^{-1}(Q)$ is a path in $H$ from $v$ to $z \in M$, and $g_i^{-1}(Q) \subseteq g_i^{-1}(V(H') \cap V(H_i)) \subseteq V(F)$, as required. 
    
    Now we show that $d_F(v) \geq 2$ for every $v \in V(F)$. If $v \in M$ then this holds because $d_F(v) \geq d_M(v) \geq 2$. Now suppose that $v \in V(H) \setminus M$. Since $v \in V(F)$, there exists some $i \in [n]$ such that $g_i(v) \in V(H')$. We have $\delta(H') \geq 2$ (because $H'$ is a copy of $H$). Also, since $v \neq z$, 
    all neighbors of $g_i(v)$ (in the graph $B$) are inside $H_i$. Hence, 
    $|N(g_i(v)) \cap V(H') \cap V(H_i)| \geq 2$. 
    Applying $g_i^{-1}$ to the set 
    $N(g_i(v)) \cap V(H') \cap V(H_i)$, we get a set of at least two neighbors of $v$ in $F$, as required. 
\end{proof}
\noindent
As in the proof of Theorem \ref{theorem:connected H, delta >= 2, type 1}, we now consider the following two cases:
    
    \paragraph{Case 1:} $F \neq H$; namely, $V(F) \subsetneq V(H)$. In this case we apply the induction hypothesis. Note that all conditions of Theorem \ref{theorem:connected H, delta >= 2, type 2} are satisfied. Indeed, $F$ is connected and $\delta(F) \geq 2$ (by Claim \ref{claim:type 2, F properties}); $M$ is a leaf biconnected component of $F$, and if $F \neq M$ then $z$ separates $M$ from $V(F) \setminus M$; and $s,a,b \in M$ is a good choice for $M$ with $z \neq s$ and 
    $z \in \{a,b\}$. 
 
    Observe that the construction $B_F := B(G,F,M,s,a,b)$ can be identified with the subgraph of 
    $B = B(G,H,M,s,a,b)$ induced by 
    $$
    V(A) \cup
    \bigcup_{i \in [n]} g_i(V(F) \cap V(N)),
    $$
    where $A = A(G,M,s,a,b)$. 
    Moreover, by the definition of the graph $F$, it holds that 
    $V(H') \subseteq V(B_F)$. 
    Now, let $f : H \rightarrow H'$ be an isomorphism from $H$ to its copy $H'$, and let $F' := f(F)$ be the copy of $F$ in $H'$. Then $F'$ is a copy of $F$ in $B'_F := B_F - \{s'w_i : i \in C\}$, contradicting the induction hypothesis. This completes the proof in Case 1.

        \paragraph{Case 2:} $F = H$. 
    By the definition of $F$, this means that for each $v \in V(H) \setminus M$, there exists some $i \in [n]$ such that $g_i(v) \in V(H')$. 
    
    \begin{claim}\label{claim:unique-i-1-connected-case-2}
        There exists at most one $i \in [n]$ such that $g_i(V(H) \setminus M) \cap V(H') \neq \emptyset$.
    \end{claim}

    \begin{proof}
        Let $I$ be the set of all $i \in [n]$ such that 
        $g_i(V(H) \setminus M) \cap V(H') \neq \emptyset$. 
        To prove that $|I| \leq 1$, we will show that 
        if $|I| \geq 2$ then $H'$ contains more 1-cuts (i.e., cut vertices) than $H$, which is impossible as $H'$ is a copy of $H$. Let $W$ be the set of all $w \in V(H)$ which are cut vertices of $H$. In particular, $z \in W$. 
        We will prove the following two claims:
        \begin{enumerate}
            \item[(i)] For every $w \in W \setminus \{z\}$, there exists $i \in [n]$ such that $g_i(w)$ is a cut vertex of $H'$. 
            \item[(ii)] For every $i \in I$, 
            $g_i(z)$ is a cut vertex of $H'$. 
        \end{enumerate}
        Note that all cut vertices of $H'$ obtained by Items (i)-(ii) are distinct, because the copies $H_1,\dots,H_n$ are pairwise vertex-disjoint and the maps $g_i : N \rightarrow H_i$ are bijective.  
        Thus, combining (i) and (ii), we conclude that $H'$ has at least $(|W|-1) + |I|$ cut vertices. On the other hand, $H'$ has exactly $|W|$ cut vertices (since $H' \cong H$). Thus, $|I| \leq 1$, as required. 

        It remains to prove (i) and (ii). We begin with (i). Let 
        $w \in W \setminus \{z\}$, and let $p \in V(H)$ be such that $M$ and $p$ are in different components of $H - w$. Note that 
        $w \notin M$ (since $z$ is the only cut vertex of $H$ contained in $M$). Also, since $p \notin M$, there exists $i \in [n]$ such that 
        $g_i(p) \in V(H')$ (as explained above). 
        Note that 
        $V(H') \setminus V(H_i) \neq \emptyset$ (as $|V(H_i)| < |V(H')|$), and that $w_i = g_i(z)$ is the only vertex in $H_i$ with neighbors outside of $H_i$ (in the graph $B$). 
        Now, since $H'$ is connected, it follows that 
        $g_i(z) \in V(H')$. We claim that every $g_i(z),g_i(p)$-path in $B$ (and thus, in $H'$) contains $g_i(w)$. Indeed, let $Q$ be such a path. As $g_i(z)$ is the only vertex in $H_i$ with neighbors outside of $H_i$, we have $Q \subseteq H_i$ and so $g_i^{-1}(Q) \subseteq 
        H - (M \setminus \{z\}) \subseteq H$ is a $p,z$-path in $H$. But $w$ separates $p,z$ in $H$ (as $z \in M$), and so $w \in g_i^{-1}(Q)$, giving $g_i(w) \in Q$. Thus, $g_i(w)$ separates $g_i(p)$ from $g_i(z)$ in $H'$, as desired.

        For (ii), fix any $i \in I$. As in the previous paragraph, we have 
        $V(H') \setminus V(H_i) \neq \emptyset$, 
        and $w_i = g_i(z)$ separates $g_i(V(H) \setminus M)$ from $V(H') \setminus V(H_i)$. Since
        $g_i(V(H) \setminus M) \cap V(H') \neq \emptyset$ (as $i \in I$), it follows that $g_i(z)$ is a separator of $H'$, as required. 
\end{proof}

 We now use Claim \ref{claim:unique-i-1-connected-case-2} to obtain a contradiction. 
    Let $\mathcal{M}(H)$ denote the set of copies of $M$ in $H$, and let $\mathcal{M}(H')$ denote the set of copies of $M$ in $H'$. 
    \begin{claim}\label{claim:type 2, M-copies}
        $|\mathcal{M}(H)| > |\mathcal{M}(H')|$.
    \end{claim}
    \begin{proof}
        First, since $s,a,b$ is a good choice for $M$, the graph 
        $A' := A - \{s' w_i : i \in C\}$ is $M$-free. 
        Hence, $H'$ has no copies of $M$ which are contained in $V(A') = V(A)$. Also, by Claim
        \ref{claim:unique-i-1-connected-case-2}, there is at most one 
        $i \in [n]$ such that $H'$ contains vertices from $g_i(V(H) \setminus M)$. 
        This means that 
        $V(H') \subseteq 
        V(A) \cup g_i(V(H) \setminus M)$.
        Note that $V(H') \not\subseteq g_i(V(H) \setminus M)$ because $|V(H) \setminus M| < |V(H)|$.
        Moreover, $g_i(z)$ separates the sets $V(A)$ and 
        $g_i(V(H) \setminus M)$ in $B$ (and hence also in $H'$, in the case that $H'$ contains vertices from $g_i(V(H) \setminus M)$). Since $M$ is 2-connected, every copy of $M$ in $H'$ is contained in 
        $g_i(V(H) \setminus (M \setminus \{z\}))$. But 
        $V(H) \setminus (M \setminus \{z\})$ clearly contains at most $|\mathcal{M}(H)| - 1$ copies of $M$ (because $H$ contains $M$ itself). This proves the claim. 
    \end{proof}
    \noindent
    Claim \ref{claim:type 2, M-copies} contradicts the fact that $H'$ is isomorphic to $H$, completing the proof of the theorem. 
\end{proof}



\section{Putting it all together}\label{sec:main}
In this section we prove Theorem \ref{thm:main}. We start with the case that $H$ is connected and $\delta(H) \geq 2$. This proof combines all of the results of the previous sections.

\begin{proof}[Proof of Theorem \ref{thm:main} for connected $H$ with $\delta(H) \geq 2$]

We will establish a reduction from the vertex cover problem for graphs with girth larger than $2v(H)$ to the problem of computing $\Rem_H(\cdot)$.

Let $M$ be a leaf biconnected component of $H$, and, if $M \neq H$, let $z$ be the cut vertex separating $M$ from $V(H) \setminus M$. Let $s,a,b \in M$ be a good choice for $M$ with $s \neq z$ (if $z$ is defined); such a choice exists by Theorem \ref{theorem:2-connected H}. We consider the following two cases:

\paragraph{Case 1:} $H = M$ or 
$z \notin \{a,b\}$. Let $G$ be an input to the vertex-cover problem, where 
$V(G) = [n]$ and $\girth(G) > 2v(H)$. Let 
$A := A(G,H,s,a,b)$ be given by Construction \ref{construction:general-construction}. 
Note that $v(A) \leq 1+n+e(G) \cdot (v(H)-3) = \poly(n)$.
We claim that $\Rem_H(A) = \tau(G)$. First, Lemma \ref{lemma:easy-inequality} gives 
$\Rem_H(A) \geq \tau(G)$. In the other direction, let $C$ be a vertex-cover of $G$ of size $|C| = \tau(G)$. 
We claim that $A' := A - \{s'w_i : i \in C\}$ is $H$-free. This would imply that $\Rem_H(A) \leq |C| = \tau(G)$, as required. So suppose by contradiction that $A'$ contains a copy $H'$ of $H$. By Lemma \ref{thm:g-locally-a-forest}, there is a subgraph $F \subseteq G$ with 
$v(F) \leq 2v(H)$
such that 
$H'$ is a subgraph of 
$A_F = A(F, H, s, a, b)$. 
(Recall that we identify $A_F$ with the subgraph of $A$ defined in Fact \ref{fact:subgraph}.)
In fact, $H'$ is a subgraph of 
$A'_F := A_F - \{s'w_i : i \in C \cap V(F)\}$, because $H'$ is a subgraph of $A'$. 
Also, $F$ is a forest (because $\girth(G) > 2v(H)$). 
But $C \cap V(F)$ is a vertex-cover of $F$, and so $A'_F$ is $H$-free, by Theorem \ref{theorem:connected H, delta >= 2, type 1}. This contradiction completes the proof in Case 1.

\paragraph{Case 2:} $z \in \{a,b\}$. Let $G$ be an input to the vertex-cover problem with 
$V(G) = [n]$ and $\girth(G) > 2v(H)$. Let 
$B := B(G,H,M,s,a,b)$ be given by Construction \ref{construction:second construction}. 
Note that 
$$
v(B) \leq 1+n+e(G) \cdot \nolinebreak (v(H)-3) + n \cdot v(H) = \poly(n).
$$
We claim that $\Rem_H(B) = \tau(G)$. First, Lemma \ref{lemma:easy-inequality, type 2} gives 
$\Rem_H(B) \geq \tau(G)$. In the other direction, let $C$ be a vertex-cover of $G$ of size $|C| = \tau(G)$. 
We claim that $B' := B - \{s'w_i : i \in C\}$ is $H$-free. This would imply that $\Rem_H(B) \leq |C| = \tau(G)$, as required. 
So suppose by contradiction that $B'$ contains a copy $H'$ of $H$. By Lemma \ref{thm:g-locally-a-forest, type 2}, there is a subgraph 
$F \subseteq G$ with 
$v(F) \leq 2v(H)$
such that 
$H'$ is a subgraph of 
$B_F := B(F,H,M,s,a,b)$. (Recall that we identify $B_F$ with the subgraph of $B$ defined in Fact \ref{fact:subgraph, type 2}.)
In fact, $H'$ is a subgraph of
$B'_F := B_F - \{s'w_i : i \in C \cap V(F)\}$, because $H'$ is a subgraph of $B'$. 
Also, $F$ is a forest (because $\girth(G) > 2v(H)$).
But $C \cap V(F)$ is a vertex-cover of $F$, and so 
$B'_F$ is $H$-free, by Theorem \ref{theorem:connected H, delta >= 2, type 2}. This contradiction completes the proof in Case 2, and hence the whole proof of the theorem.
\end{proof}

\noindent
In the rest of this section, we first prove Theorem \ref{thm:main} for all connected graphs $H$ containing a cycle (Section \ref{subsec:connected, delta=1}) and then finally in full generality (Section \ref{subsec:disconnected}).

\subsection{Theorem \ref{thm:main} for connected graphs $H$}\label{subsec:connected, delta=1}

Here we prove Theorem \ref{thm:main} for every connected graph $H$ containing a cycle. We already proved the theorem in the case that $\delta(H) \geq 2$, and we will now reduce the general case to this one. The key notion is that of a {\em 2-core} of a graph. 

\begin{definition}[2-core]\label{def:2-core}
The {\em 2-core} of a graph $H$, denoted $H_*$, is the maximal subgraph of $H$ with minimum degree at least 2.
\end{definition}

\noindent
It is easy to see (and well-known) that the 2-core is unique, and it can be obtained from $H$ by repeatedly deleting vertices of degree at most 1 until no such vertices remain. Also, if $H$ is connected and $H_*$ is non-empty then $H_*$ is also connected; and $H_*$ is non-empty if and only if $H$ contains \nolinebreak a \nolinebreak cycle. 

\begin{lemma}\label{lemma:min-degree-2-contained-in-2-core}
    If a graph $H$ is a subgraph of a graph $G$ and 
    $\delta(H) \geq 2$, then $H$ is a subgraph of $G_*$.
\end{lemma}
\begin{proof}
    If not, then consider the graph $H \cup G_*$. This is a strict supergraph of $G_*$, and hence has $\delta(H \cup G_*) \leq 1$ (by the maximality of $G_*$). But every vertex of $G_*$ has degree $\geq 2$ in $G_*$ and hence in 
    $H \cup G_*$. Therefore, there exists 
    $v \in V(H)$ with $d_H(v) \leq d_{H \cup G_*}(v) \leq 1$, contradicting $\delta(H) \geq 2$.
\end{proof}

\begin{theorem}\label{theorem:2-core reduction}
Let $H$ be a connected graph containing a cycle and let $H_*$ be its 2-core. Then computing $\Rem_{H_*}(\cdot)$ is polynomially reducible to computing $\Rem_H(\cdot)$.
\end{theorem}

\begin{corollary}\label{cor:connected H}
    Theorem \ref{thm:main} holds for all connected graphs $H$ containing a cycle. 
\end{corollary}

\begin{proof}[Proof of Theorem \ref{theorem:2-core reduction}]
Recall that $H_*$ is obtained from $H$ by deleting vertices of degree at most 1, as long as such vertices exist. Hence, $H$ is obtained from $H_*$ by ``gluing" certain trees on the vertices of $H^*$. More precisely, writing $V(H_*) = \{v_1,\dots,v_h\}$, there are rooted trees $T_1,\dots,T_k$ and a function $\varphi : [k] \rightarrow [h]$ such that $H$ is obtained from $H_*$ by adding, for each $i \in [k]$, a copy of $T_i$ whose root is $v_{\varphi(i)}$ and whose other vertices are new. 

Let $G$ be an $n$-vertex input for the problem $\Rem_{H_*}(\cdot)$. Let $G'$ be the graph obtained from $G$ by gluing $e(G)+1$ copies of $T_i$ at every vertex of $G$, for every $i \in [k]$. Then 
$$
v(G') \leq n + n \cdot (e(G)+1) \cdot \sum_{i=1}^k v(T_i) = O(n^3) = \poly(n).
$$
The following proves the validity of this reduction.
\begin{claim}\label{claim:np-hard-min-degree-1}
        $\Rem_{H}(G') = \Rem_{H_*}(G)$.
    \end{claim}
    
    \begin{proof}
        Let $E \subseteq E(G)$ be a subset of edges such that $G - E$ is $H_*$-free. We claim that $G' - E$ is $H$-free. Indeed, suppose for the sake of contradiction that $H'$ is a copy of $H$ in $G' - E$. 
        Let $H_*'$ denote the copy of $H_*$ in $H'$. 
        By Lemma \ref{lemma:min-degree-2-contained-in-2-core}, $H_*'$ is contained in the 2-core $G_*'$ of $G'$ (because $\delta(H_*') \geq 2$). But since $G'$ is obtained from $G$ by gluing trees, it holds that 
        $G_*' \subseteq G$. Hence, $H_*'$ is a subgraph of $G$, and thus of $G-E$. But $G-E$ is $H_*$-free, a contradiction.  
        This shows that 
        \begin{equation}\label{eq:2-core reduction, easy direction}
        \Rem_{H}(G') \leq \Rem_{H_*}(G)
        \end{equation} 
        
        To show the other inequality, let 
        $E' \subseteq E(G')$ be a 
        minimum-size subset of $E(G')$ such that $G' - E'$ is $H$-free. 
        Crucially, note that $|E'| \leq \Rem_{H_*}(G) \leq e(G)$, due to \eqref{eq:2-core reduction, easy direction}.
        We claim that $G-E'$ is $H_*$-free. So suppose otherwise, and let $H_*'$ be a copy of $H_*$ in $G-E'$. Consider any $v \in H_*'$ and any $i \in [k]$. Since $|E'| \leq e(G)$, one of the $e(G)+1$ copies of $T_i$ which were glued on $v$ contains no edges from $E'$. Hence, we can choose copies of $T_1,\dots,T_k$ (each glued on an appropriate vertex $v \in V(H_*')$) to obtain a copy of $H$ in $G'-E'$, a contradiction to the choice of $E'$. This proves our claim that $G-E'$ is $H_*$-free. It follows that 
        $\Rem_{H_*}(G) \leq |E'| = \Rem_{H}(G')$, as required. 
\end{proof}

\noindent
Claim \ref{claim:np-hard-min-degree-1} completes the proof of the theorem.
\end{proof}

\subsection{Disconnected graphs}\label{subsec:disconnected}

So far we proved Theorem \ref{thm:main} for all connected graphs $H$ (containing a cycle). Here we extend this to disconnected graphs as well, by showing, in a sense, that a graph $H$ is at least as hard as its hardest connected component. 
This is essentially a known reduction, and we follow the presentation from \cite{gishboliner2024trimming}. For a graph $G$ and integer $k$, denote by $kG$ the disjoint union of $k$-many copies of $G$. 
The following lemma is equivalent to 
\cite[Lemma 5.1]{gishboliner2024trimming}.\footnote{Lemma 5.1 in \cite{gishboliner2024trimming} was formulated in terms of the parameter
$\ex(G, H) := e(G) - \Rem_H(G)$. In other words, $\ex(G,H)$ is the maximum number of edges in an $H$-free spanning subgraph of $G$. It is easy to see that Lemma \ref{lemma:np-hard-containing-cycle-1} is equivalent.}

\begin{lemma}\label{lemma:np-hard-containing-cycle-1}
    Let $C$ be a connected graph with $v(C) \geq 2$, and let $k \geq 2$. Then, for every graph $G$, 
    $$
    \Rem_{kC}(kG) = \Rem_{kC}((k - 1)G) + 
    \Rem_{C}(G).
    $$
\end{lemma}

\begin{lemma}\label{lemma:np-hard-containing-cycle-2}
    For each connected $H$ containing a cycle and every $k \geq 1$, computing $\Rem_{kH}( \cdot )$ is NP-hard.
\end{lemma}

\begin{proof}
    The case $k = 1$ follows by Corollary \ref{cor:connected H}, so suppose $k \geq 2$. By Lemma \ref{lemma:np-hard-containing-cycle-1}, if we could compute $\Rem_{kH}(G)$ in polynomial time for every graph $G$, then we could compute $\Rem_H(G)$ in polynomial time for every graph $G$, as
    $$
    \Rem_{H}(G) = \Rem_{kH}(kG) - 
    \Rem_{kH}((k - 1)G).
    $$
    But computing $\Rem_{H}(G)$ is NP-hard by Corollary \ref{cor:connected H}.
\end{proof}

\begin{proof}[Proof of Theorem \ref{thm:main}]
    Let $H$ be a graph containing a cycle, 
    and let $C_1, \dots, C_{\ell}$ be the connected components of $H$. Suppose without loss of generality that $C_1$ contains a cycle and contains the largest number of edges among the components which contain a cycle. By permuting the indices, we can also assume that $C_1, \dots, C_k$ are isomorphic to $C_1$, and 
    $C_{k + 1}, \dots, C_{\ell}$ are not isomorphic to $C_1$ (for some $1 \leq k \leq \ell$). Note that $C_1$ is not a subgraph of $C_i$ for any $i \geq k+1$, because 
    either $C_i$ has no cycles (and then clearly $C_1 \not\subseteq C_i$), or else 
    $e(C_i) \leq e(C_1)$ and $C_i$ is not isomorphic to $C_1$. 
    
    We reduce the problem of computing 
    $\Rem_{kC_1}( \cdot )$ to the problem of computing $\Rem_{H}( \cdot )$, the former being NP-hard by Lemma \ref{lemma:np-hard-containing-cycle-2}. Let $G$ be an input graph with $n$ vertices. Let $G'$ be the graph obtained from $G$ by adding to it, for each $k + 1 \leq i \leq \ell$, a collection of $e(G)+1$ disjoint copies of $C_i$ which are also disjoint from $G$. 
    Then $v(G') \leq n + (e(G)+1) \cdot v(H) = O(n^2) = \poly(n)$.
    We will show that
    $\Rem_{H}(G') = \Rem_{kC_1}(G),$
    which will prove the correctness of the reduction. Observe that if we destroy all copies of $kC_1$ in $G'[V(G)]$ then the resulting subgraph of $G'$ is $kC_1$-free, since $C_1$ is connected and none of $C_{k + 1}, \dots, C_{\ell}$ contain $C_1$ as a subgraph. In particular, this gives an $H$-free subgraph of $G'$, showing that 
    $\Rem_{H}(G') \leq \Rem_{kC_1}(G).$
    Note that in particular,
    $\Rem_{H}(G') \leq e(G)$. 

    To show the other direction, let $E' \subseteq E(G')$ be a minimum-size edge-set such that $G'-E'$ is $H$-free. By the above, $|E'| \leq e(G)$. Hence, for each 
    $k + 1 \leq i \leq \ell$, $E'$ misses (at least) one of the $e(G)+1$ disjoint copies of $C_i$ added to $G'$. But then $G - E'$ must be $kC_1$-free, because otherwise $G'-E'$ would contain a copy of $H$. 
    Hence, 
    $\Rem_{H}(G') = |E'| \geq \Rem_{kC_1}(G)$
    as desired, completing the proof.
\end{proof}

\paragraph{Acknowledgments:}
The first author thanks Asaf Shapira and Yevgeny Levanzov for very useful conversations on the topic of this paper.

\bibliographystyle{abbrv}
\bibliography{bibfile}

\end{document}